\documentclass[a4paper, 12pt]{article}
\newcommand{\toukou}[1]{\ifx\TOUKOU\undefined\else{#1}\fi}%
\newcommand{\toukoudel}[1]{\ifx\TOUKOU\undefined{#1}\else\fi}%
\newcommand{\toukouchange}[2]{\ifx\TOUKOU\undefined{#1}\else{#2}\fi}%
\usepackage{amssymb}
\usepackage{amsmath}
\usepackage[noadjust]{cite}
\usepackage[dvipdfm]{graphicx} 
\usepackage{amsthm}
\usepackage{verbatim,enumerate}
\usepackage{ascmac}
\usepackage{enumerate}
\usepackage{fancybox} 
\usepackage[active]{srcltx}
\toukoudel{
 \newtheorem{theorem}{Theorem}[section]
 
 \newtheorem{lemma}[theorem]{Lemma}
 
\theoremstyle{definition}
 \newtheorem{definition}[theorem]{Definition}

}
 \newtheorem{fact}[theorem]{Fact}
\numberwithin{equation}{section}
\renewcommand{\theenumi}{{\rm(\arabic{enumi})}}

\usepackage[usenames]{color}

\newcommand{\R}{\boldsymbol{R}}

\newcommand{\rank}{\operatorname{rank}}

\renewcommand{\phi}{\varphi}

\newcommand{\hess}{\operatorname{Hess}}

\newcommand{\A}{\mathcal{A}}

\newcommand{\pmt}[1]{{\begin{pmatrix} #1  \end{pmatrix}}}

\newcommand{\mycomment}[1]{}
\renewcommand{\tilde}{\widetilde}
\newcommand{\coef}{\operatorname{coef}}
\newcommand{\chartonoff}[1]{#1}

\begin{document}

\title{Criteria for codimension two singularities
of surfaces and their applications}

\author{Kentaro Saji and Runa Shimada}
\toukou{
\address{
Department of Mathematics,
Graduate School of Science,
Kobe University, 
1-1, Rokkodai, Nada, Kobe 
657-8501, Japan,
E-mails: {\tt saji@math.kobe-u.ac.jp},
{\tt 231s010s@stu.kobe-u.ac.jp}
}
}
\maketitle
\begin{abstract}
We give simple criteria for the singularities 
 appearing on surfaces codimension less than or equal to two.
As applications, we give 
conditions for codimension two
singularities that appear in ruled surfaces and center maps of surfaces in the Euclidean space.
\end{abstract}
\toukouchange
{
\renewcommand{\thefootnote}{\fnsymbol{footnote}}
}
{
}

\section{Introduction}
Singularities of smooth map-germs have been studied by many authors,
up to the
equivalence under coordinate changes
in both source and target. This equivalence relation is 
called the $\A$-equivalence.
The simple singularities appearing on surfaces in the three space
are classified with respect to this equivalence in \cite{mond}.
\begin{fact}{\rm \cite[Theorem 1.1]{mond}}\label{fact:mond}
A map-germ\/ $f:(\R^2,0)\to(\R^3,0)$ has a singular point at\/ $0$, and
the codimension\/ $($codim\/$)$ is less than two, then\/
$f$ is\/ $\A$-equivalent to one of the following.
\begin{table}[h!]
\centering
\begin{tabular}{ccc}
\hline
Germ&codim&name\\
\hline
$(u,v^2,uv)$&0&Whitney umbrella $(S_0)$\\
$(u,v^2,v(\pm u^2+v^2))$&1& $S_1^\pm$\\
$(u,v^2,v(u^3+v^2))$&2& $S_2$\\
$(u,v^2,v(u^2\pm v^4))$&2& $B_2^\pm$\\
$(u,uv+v^5,v^3)$&2& $H_2$\\
\hline
\end{tabular}
\end{table}
\end{fact}
\noindent 
In \cite[Theorem 1.1]{mond},
a classification of simple singularities is given,
however we reproduce here only singularities with codimension up to two,
which are dealt with this paper.
The $S_1^\pm$ singularities are also called the
{\it Chen-Matumoto-Mond\/ $\pm$ singularities}, see \cite{chenmatumoto}.

On the other hand, the question which germ on the classification table
is a given germ equivalent to? is another problem,
which is called recognition.
In this paper, we give simple criteria for the singularities 
up to codimension two for recognition.
In the previous method used for recognition is
first normalize the given germ and then its jet is studied. 
The criteria of singularities without using
normalization are not only more convenient
but also indispensable in some cases,
where we will describe the conditions for such singularities
in ruled surfaces and center maps.
See \cite{KRSUY,sajipl,ak} and \cite{irrtbook,usybook} for
criteria for other several singularities
and its applications.
As applications of the criteria, we give 
conditions for codimension less than or equal to two
singularities that appear in  ruled surfaces and center maps of surfaces in the Euclidean space.

\section{Criteria}
Let $f,g:(\R^2,0)\to(\R^3,0)$ be two $C^\infty$ map-germs. Then
$f$ and $g$ are said to be 
$\A$-{\it equivalent\/} if there exist diffeomorphism-germs
$\phi_s:(\R^2,0)\to(\R^2,0)$ and $\phi_t:(\R^3,0)\to(\R^3,0)$
such that
$\phi_t\circ f\circ \phi_s^{-1}=g$ holds.
A map-germ $f:(\R^2,0)\to(\R^3,0)$ is an {\it $S_2$ singularity\/} if
it is $\A$-equivalent to the germ $(u,v)\mapsto(u,v^2,v(u^3+v^2))$
at the origin $0=(0,0)$.
We define other singularities that appear in Fact \ref{fact:mond} 
in this manner by the formulas and the names given in the table.

Throughout this paper, we assume
$f:(\R^2,0)\to(\R^3,0)$ is a $C^\infty$-map satisfying
$\rank df_0=1$.
Then there exists a pair of linearly independent vector fields
$(\xi,\eta)$
such that $\eta$ generates $\ker df_0$.
A pair of vector fields $(\xi,\eta)$ is said to be {\it adapted\/}
if 
they are linearly independent and $\eta$ generates $\ker df_0$.
A coordinate system is said to be {\it adapted\/}
if $(\partial_u,\partial_v)$ is adapted.
The adaptivity is defined for a pair of vector fields
$(\xi,\eta)$, however we omit the word ``a pair of'' for
short.
We show the following:
\renewcommand{\theenumi}{$(\mathrm{\Roman{enumi}})$}
\begin{theorem}\label{thm:cri}
Let\/ $f:(\R^2,0)\to(\R^3,0)$ satisfy\/
$\rank df_0=1$.
Then\/ $f$ is 
\begin{enumerate}
\item\label{itm:cri1} an\/ $S_2$ singularity if and only if\/
$f$ is of\/ $S$-type, and for an\/ $S$-$3$-adapted vector fields\/
$(\xi,\eta)$, it holds that
\begin{equation}\label{eq:cri1}
\det(\xi f,\xi^3\eta f,\eta^2f)(0)\ne0.
\end{equation}
\item\label{itm:cri2}
a\/ $B_2^\pm$ singularity if and only if\/
$f$ is of\/ $B$-type, and for a\/ $B$-$3$-adapted vector fields\/
$(\xi,\eta)$, it holds that
\begin{equation}\label{eq:cri2}
-5\det(\xi f,\eta^2 f,\eta^3\xi f)(0)^2
+
3\det(\xi f,\eta^2 f,\eta\xi^2f)(0)
 \det(\xi f,\eta^2f,\eta^5f)(0)\ne0.
\end{equation}
Moreover, the\/ $\pm$-sign of\/ the $B_2^\pm$ singularity 
is determined by the sign of\/ \eqref{eq:cri2}.
\item\label{itm:cri3}
an\/ $H_2$ singularity if and only if\/
$f$ is of\/ $H$-type, and for an\/ $H$-$4$-adapted vector fields\/
$(\xi,\eta)$, it holds that
\begin{equation}\label{eq:cri3}
\det(\xi f,\eta^5f,\eta^3f)(0)\ne0.
\end{equation}
\end{enumerate}
\end{theorem}
Here, the directional derivative of a map $f$ by a vector field $\zeta$
is denoted by $\zeta f$, and $n$ times derivatives by $\zeta$
is denoted by $\zeta^n f$.
\begin{lemma}
Let\/ $(u,v)$ is an adapted coordinate system. 
Then another coordinate system\/ $(x,y)=(x(u,v),y(u,v))$ is adapted
if and only if\/ $x_v(0)=0$.
Let\/ $(\xi,\eta)$ is an adapted vector fields.
Then another pair of vector fields\/ $(\tilde \xi,\tilde \eta)$, where\/
$\tilde \xi=a\xi+b\eta$,
$\tilde \eta=c\xi+d\eta$
is adapted if and only if\/
$c(0)=0$. 
The existence of an adapted vector fields implies
the existence of an adapted coordinate system.
\end{lemma}
\begin{proof}
We show the last claim, where the other claims are obvious.
The existence of an adapted vector fields implies
$\rank df_0=1$, and this property implies
the existence of an adapted coordinate system.
\end{proof}
Let $(\xi,\eta)$ be an adapted vector fields.
The condition
$\xi f\times \eta\eta f\ne0$ at $0$ does not depend on the choice of
adapted vector fields and the diffeomorphism on the target space.
Moreover, 
the condition $\xi f\times \eta\eta f=0$ and
$\xi f\times \eta\xi f\ne0$ at $0$
does not depend on the choice of
adapted vector fields and the diffeomorphism on the target space.
\begin{definition}
A map-germ $f:(\R^2,0)\to(\R^3,0)$ is said to be {\it $SB$-type\/}
if $\xi f\times \eta\eta f\ne0$ at $0$ holds.
A map-germ $f$ is said to be {\it $HP$-type\/}
if $\xi f\times \eta\eta f=0$ and
$\xi f\times \eta\xi f\ne0$ at $0$ 
holds.
\end{definition}
Germs whose two jets are $\A$-equivalent to
$(u,uv,0)$ may be $\A$-equivalent to $H$ or $P$ singularities
and others in the classification table in \cite[Table 1]{mond}.
The condition
$\det(\xi f,\eta^2f,\xi\eta f)\ne0$ at $0$ 
does not depend on the choice of
adapted vector fields and the diffeomorphism on the target space.
If this condition holds, then $f$ is the Whitney umbrella.
We set $\phi(\xi,\eta)=\det(\xi f,\eta f,\eta^2f)$.
We write $\phi=\phi(\xi,\eta)$ if the vector fields
$(\xi,\eta)$ what we are taking are obvious.
Then the above condition is equivalent to $d\phi_0\ne0$,
since $\eta\phi(0)=0$ holds.
The map-germ $f$ is an $S_1^-$ singularity if and only if
$d\phi_0=0$, $\det\hess\phi(0)<0$ and $\eta^2 \phi(0)\ne0$,
and
$f$ is an $S_1^+$ singularity if and only if
$d\phi_0=0$, $\det\hess\phi(0)>0$.
We remark that in the both cases,
$\eta^2 \phi(0)\ne0$ holds.
The two jet $j^2f(0)$ of a map-germ $f$ satisfying $\rank df_0=1$
is $\A$-equivalent to one of
$(u,v^2,uv)$,
$(u,v^2,0)$,
$(u,uv,0)$ or
$(u,0,0)$ up to two jet.
If $f$ is the Whitney umbrella, then it is equivalent to the first one,
if $f$ is $SB$-type, then it does the second one,
if $f$ is $HP$-type, then it does the third one.
If $f$ is neither $SB$-type nor $HP$-type, then 
it does the fourth one.
The fourth one is not dealt with this paper.
\subsection{Preliminaries for brunch to\/ $S_2$ or\/ $B_2$ singularities}
We assume that the map $f$ is $SB$-type and
not a Whitney umbrella. 
Then 
$\xi f\times \eta^2 f\ne0$ and
$\det(\xi f,\eta^2f,\xi\eta f)(0)=0$ at $0$
hold.

\begin{lemma}\label{lem:vfexist01}
Let $f$ be $SB$-type and
not a Whitney umbrella. 
$(1)$
there exists an adapted coordinate system\/ $(u,v)$ such that\/
$f_{uv}(0)=0$ holds.
$(2)$
there exists an adapted vector fields\/ $(\xi,\eta)$ such that\/
$\xi\eta f(0)=\eta\xi f(0)=0$ holds.
\end{lemma}
The assertion $(1)$ implies $(2)$.
However, we give proof for each assertion since 
we would like to give a concrete construction
for desired vector fields.
\begin{proof}
(1) Let $(x,y)$ be an adapted coordinate system.
The condition
$\xi f\times \eta^2 f\ne0$ and
$\det(\xi f,\eta^2,\xi\eta f)(0)=0$ at $0$
does not depend on the choice of adapted vector fields,
the same condition holds for the vector fields
$(\partial_x,\partial_y)$.
Thus we set
$f_{xy}=\alpha f_x+\beta f_{yy}$ at $0$.
By a coordinate change
$x=u-\alpha uv$, 
$y=-\beta u+v$, we see
$$
f_{uv}=-\alpha f_x-\beta f_{yy}+f_{xy}\quad\text{at}\quad0.
$$
Substituting $f_{xy}=\alpha f_x+\beta f_{yy}$, we have the assertion.
(2) Let $(\xi,\eta)$ be an adapted vector fields.
We take a coordinate system $(u,v)$ satisfying
$\xi=\partial_u,\eta=\partial_v$ at $0$.
Then $(\xi,\eta)$ is an adapted vector fields,
and the condition
$\xi f\times \eta^2 f\ne0$,
$\det(\xi f,\eta^2,\xi\eta f)(0)=0$ at $0$
does not depend on the choice of adapted vector fields,
it holds that
$f_u\times f_{vv}\ne0$,
$\det(f_u,f_{vv},f_{uv})(0)=0$ at $0$.
We set
\begin{equation}\label{eq:sbadvf}
\tilde\xi=a\partial_u+b\partial_v,\quad
\tilde\eta=c\partial_u+d\partial_v,
\end{equation}
where
\begin{equation}\label{eq:sbadvf2}
a=1-\alpha v,\quad
b=-\beta,\quad
c=-\alpha u,\quad
d=1,
\end{equation}
for 
$
f_{uv}=\alpha f_u+\beta f_{vv}$ at $0$.
Then $(\tilde\xi,\tilde\eta)$ is adapted, and
$
\tilde\xi\tilde\eta f=\tilde\eta\tilde\xi f=0$ at $0$
holds.
\end{proof}

\begin{definition}
An adapted vector fields $(\xi,\eta)$ is said to be
$SB$-$2$-{\it adapted\/} if $\xi\eta f=\eta\xi f=0$ at $0$
holds.
A coordinate system $(u,v)$ is said to be $SB$-$2$-{\it adapted\/} 
if $(\partial_u,\partial_v)$ is $SB$-$2$-adapted.
\end{definition}
\begin{lemma}\label{lem:sb2adcoordex}
The existence of an\/ $SB$-$2$-adapted vector fields\/ $(\xi,\eta)$
implies the existence of an\/ $SB$-$2$-adapted coordinate system\/ $(u,v)$.
\end{lemma}
Multiplying an $SB$-$2$-adapted vector fields $(\xi,\eta)$ 
by a non-zero function, then it is not an $SB$-$2$-adapted vector fields.
Multiplying an $SB$-$2$-adapted vector fields $(\xi,\eta)$ 
by a non-zero function $k$, and taking a coordinate system $(u,v)$ such that
$\partial_u=k\xi$,
$\partial_v=k\eta$ (\cite[Lemma in p182]{kobanomi}),
then it is not $SB$-$2$-adapted.
Thus a proof is needed for this lemma.
\begin{proof}
Let $(\xi,\eta)$ be an $SB$-$2$-adapted vector fields.
Then there exists a coordinate system $(u,v)$ such that
$(\partial_u,\partial_v)=(\xi,\eta)$ holds at $0$.
This is adapted coordinate system.
Since the condition
$\xi f\times \eta\eta f=0$ and
$\det(\xi f,\eta^2f,\xi\eta f)=0$ at $0$ does not depend on
the choice of adapted vector fields,
$\det(f_u,f_{vv},f_{uv})=0$ holds.
Then by Lemma \ref{lem:vfexist01}, the
existence of 
$SB$-$2$-adapted coordinate system follows.
\end{proof}
\begin{lemma}\label{lem:2adaptedcond}
$(1)$ Let\/ $(u,v)$ be an\/ $SB$-$2$-adapted coordinate system.
Then an adapted coordinate system\/ $(x,y)=(x(u,v),y(u,v))$ 
$($i.e., $x_v=0$ holds\/$)$
is\/ $SB$-$2$-adapted if and only if
\begin{equation}\label{eq:2adaptedcond}
x_{uv}(0)=0,\quad y_u(0)=0.
\end{equation}
$(2)$ Let\/ $(\xi,\eta)$ be an\/ $SB$-$2$-adapted vector fields.
Then an adapted vector fields\/ $(\tilde \xi,\tilde \eta)$ 
$($i.e., $c=0$ holds\/$)$, where\/ 
$\tilde \xi=a\xi+b\eta$,
$\tilde \eta=c\xi+d\eta$ is\/ $SB$-$2$-adapted 
if and only if
\begin{equation}\label{eq:2adaptedcond}
\eta a=0,\ b=0,\ \xi c=0\quad\text{at}\quad 0.
\end{equation}
\end{lemma}
\begin{proof}
(1)By the $SB$-$2$-adaptivity of $(x,y)$ and $f$, it holds that
$x_v=0$, $f_v=0$, $f_{uv}=0$ at $0$.
Thus
$$f(x(u,v),y(u,v))_{uv}=
f_u x_{uv}+y_uy_v f_{vv}\quad \text{at}\quad0$$
holds, and this shows the assertion.
(2) We set
$\tilde\xi=a\xi+b\eta$,
$\tilde\eta=c\xi+d\eta$.
Then the assertion follows from the calculation
\begin{align*}
\tilde\eta\tilde\xi f=&c\xi(\tilde\xi f)
+
d(\eta a \xi f+a \eta\xi f+\eta b \eta f+b\eta^2f),\\
\tilde\xi\tilde\eta f=&
a(\xi c \xi f+c \xi^2 f+\xi d \eta f+d\eta \xi f)
+
b(\eta c \xi f+c\eta\xi f+\eta d \eta f +d \eta^2f).
\end{align*}
\end{proof}

\subsection{Branch to\/ $S_2$ or\/ $B_2$ singularities}
Let $(\xi,\eta)$ be an $SB$-$2$-adapted vector fields.
Then since
$\eta\phi(\xi,\eta)
=
\det(\eta\xi f,\eta f,\eta^2f)
+\det(\xi f,\eta f,\eta^3f)$,
it holds that $\eta^2\phi(\xi,\eta)(0)
=
\det(\xi f,\eta^2f,\eta^3f)(0)$.
\begin{lemma}\label{lem:sb2yyy}
$(1)$ Let\/ $(u,v)$ be an\/ $SB$-$2$-adapted 
coordinate system.
Then the condition\/
$(\phi(\partial_u,\partial_v))_{vv}\ne0$
does not depend on the choice of\/
$SB$-$2$-adapted coordinate system
and diffeomorphism on the target space.
$(2)$ Let\/ $(\xi,\eta)$ be an\/ $SB$-$2$-adapted
vector fields.
Then the condition\/
$\eta^2\phi(\xi,\eta)\ne0$
does not depend on the choice of\/
$SB$-$2$-adapted vector fields 
and diffeomorphism on the target space.
\end{lemma}
\begin{proof}
It is enough to show the case of vector fields.
Since the existence of $SB$-$2$-adapted vector fields implies
the existence of $SB$-$2$-adapted coordinate system
(Lemma \ref{lem:vfexist01}), we assume  $(u,v)$ 
is an $SB$-$2$-adapted coordinate system.
Let $(\xi,\eta)$ be an $SB$-$2$-adapted vector fields, where
$\xi=a\partial_u+b\partial_v$,
$\eta=c\partial_u+d\partial_v$ and
$a_v=0,b=0,c_u=0$, $ad\ne0$.
Then it holds that
$\eta^2\phi(\xi,\eta)=\det(\xi f,\eta^2f,\eta^3f)$ at $0$.
Since
$f_\xi=af_u$,
$f_{\eta^2}=c(c_uf_u+cf_{uu}+d_uf_v+df_{uv})
+d(c_vf_u+cf_{uv}+d_vf_v+df_{vv})$ and
$\eta^3 f=*f_u+*f_{vv}+d^3f_{vvv}$ holds at $0$,
we have $\eta^2\phi(\xi,\eta)=ad^5\det(f_u,f_{vv},f_{vvv})$ at $0$.
Here, the symbol $*$ stands for a term who will not be used
in the later calculations.
For the independence of target diffeomorphisms, 
see Lemma \ref{lem:target}.
\end{proof}
\begin{lemma}\label{lem:sb2xyy}
$(1)$ Let\/ $(u,v)$ be an\/ $SB$-$2$-adapted coordinate system.
Then the condition\/ $(\phi(\partial_u,\partial_v))_{uu}(0)\ne0$
is equivalent to\/
$\det(f_u,f_{uuv},f_{vv})(0)\ne0$, and it does not depend on the
choice of\/ $SB$-$2$-adapted coordinate system and
diffeomorphisms on the targe space.
$(2)$ Let\/ $(\xi,\eta)$ be an\/ $SB$-$2$-adapted vector fields.
Then the condition\/ $\xi^2\phi(\xi,\eta)(0)\ne0$ is equivalent to\/
$\det(\xi f,\xi^2\eta f,\eta^2f)(0)\ne0$,
which is equivalent to both\/
$\det(\xi f,\xi\eta\xi f,\eta^2f)(0)\ne0$ and\/
$\det(\xi f,\eta\xi^2 f,\eta^2f)(0)\ne0$.
Moreover, it does not depend on the
choice of\/ $SB$-$2$-adapted vector fields and
diffeomorphisms on the targe space.
\end{lemma}
\begin{proof}
It is enough to show the case of vector fields.
As in the proof of Lemma \ref{lem:sb2yyy},
we assume  $(u,v)$ 
is an $SB$-$2$-adapted coordinate system.
Let $(\xi,\eta)$ be an $SB$-$2$-adapted vector fields, where
$\xi=a\partial_u+b\partial_v$,
$\eta=c\partial_u+d\partial_v$ and
$a_v=0,b=0,c_u=0$, $ad\ne0$.
Then it holds that
$\xi f=af_u$,
$\eta\xi f=a(c_uf_u+cf_{uu}+d_uf_v+df_{uv})
+b(c_vf_u+cf_{uv}+d_vf_v+df_{vv})$,
$\eta^2f=c(c_uf_u+cf_{uu}+d_uf_v+df_{uv})
+d(c_vf_u+cf_{uv}+d_vf_v+df_{vv})$ and
$\eta\xi^2f=a((a_uc_u+ac_{uu}+b_uc_v)f_u+b_udf_{vv}+adf_{uuv})$,
$\eta^3 f=*f_u+*f_{vv}+d^3f_{vvv}$
at $0$.
On the other hand, we have
$\xi^2\phi(\xi,\eta)=a^3d^3\det(\xi f,\xi^2\eta f,\eta^2f)$ at $0$.
By the same calculations, we see 
$\xi^2\eta f$,
$\xi\eta\xi f$
has the form $*f_u+a^2df_{uuv}+*f_{vv}$.
Thus the independence of the choice of vector fields follows.
For the independence of target diffeomorphisms, 
see Lemma \ref{lem:target}.
\end{proof}

\begin{lemma}
$(1)$ Let\/ $(u,v)$ be an\/ $SB$-$2$-adapted coordinate system.
Then it holds that\/ 
$(\phi(\partial_u,\partial_v))_{uv}(0)=0$.
$(2)$ Let\/ $(\xi,\eta)$ be an\/ $SB$-$2$-adapted vector fields.
Then\/ $\xi\eta\phi(\xi,\eta)(0)=\eta\xi\phi(\xi,\eta)(0)=0$ holds.
\end{lemma}
\begin{proof}
It is enough to show the case of vector fields.
Since
$\eta\phi(\xi,\eta)=\det(\xi\eta f,\eta f,\eta^2 f)+
\det(\xi f,\eta f,\eta^3f)$ and $\xi\eta f=
\eta\xi f=0$ holds at $0$, 
we see $\eta\xi \phi=0$ at $0$.
Since $[\eta,\xi]$ is a vector fields, and $d\phi_0=0$ holds,
$\eta\xi \phi=0$ at $0$ also follows.
\end{proof}
Let $(\xi,\eta)$ be an $SB$-$2$-adapted vector fields.
Since
$$
\hess \phi(\xi,\eta)(0)=
\pmt{
\xi^2\phi(0)&\eta\xi\phi(0)\\
\xi\eta\phi(0)&\eta^2\phi(0)}
=
\pmt{
\xi^2\phi(0)&0\\
0&\eta^2\phi(0)},
$$
if $f$ is not an $S_1^\pm$ singularity,
$\xi^2\phi(\xi,\eta)=0$ or
$\eta^2\phi(\xi,\eta)=0$ holds.
\begin{definition}
Let $f$ be an $SB$-type map-germ.
The germ $f$ is of $S$-{\it type\/}
if $\xi^2\phi(\xi,\eta)=0$, $\eta^2\phi(\xi,\eta)\ne0$ at $0$ hold
for an $SB$-$2$-adapted vector fields $(\xi,\eta)$.
The germ $f$ is of $B$-{\it type\/}
if $\eta^2\phi(\xi,\eta)=0$, $\xi^2\phi(\xi,\eta)\ne0$ at $0$ hold
for an $SB$-$2$-adapted vector fields $(\xi,\eta)$.
\end{definition}
Since the existence of an
$SB$-$2$-adapted vector fields implies 
the existence of $SB$-$2$-adapted coordinate system $(u,v)$.
Since $(\partial_u,\partial_v)$ is an
$SB$-$2$-adapted vector fields, and the both conditions of 
$S$-type and $B$-type do not depend on the choice of
and $SB$-$2$-adapted vector fields.
Thus 
the definition of $S$-type map-germ 
defined by using an $SB$-$2$-coordinate system,
and
the definition of $S$-type map-germ
defined by using an $SB$-$2$-adapted vector fields,
are equivalent, and
the same holds for the definition of $B$-type map-germ.

\subsection{The condition for\/ $S_2$ singularity}

\begin{lemma}\label{lem:s3adapted}
Let us assume\/ $f$ is of\/ $S$-type.
$(1)$ There exists an\/ $SB$-$2$-adapted coordinate system\/ $(u,v)$
such that\/ $f_{uuv}(0)=0$.
$(2)$ There exists an\/ $SB$-$2$-adapted vector fields\/ $(\xi,\eta)$
such that\/ $\eta\xi^2f(0)=\xi\eta\xi f(0)=\xi^2\eta f(0)=0$.
\end{lemma}
\begin{proof}
(1)
By the assumption, there exists an $SB$-$2$-adapted coordinate system $(x,y)$.
Since $f$ is of $S$-type, 
$\det(f_x,f_{xxy},f_{yy})(0)=0$ holds, we set
$f_{xxy}=\alpha f_x+\beta f_{yy}$ at $0$.
Setting a coordinate system $(u,v)$ by
$x=u-\alpha u^2v/2$,
$y=v-\beta u^2/2$,
then by Lemma \ref{lem:2adaptedcond}, $(u,v)$ is of $SB$-$2$-adapted,
and
$f_{uuv}=-\alpha f_x-\beta f_{yy}+f_{xxy}$ holds.
Thus $f_{uuv}(0)=0$ holds.
(2)
Let $(\xi,\eta)$ be an adapted vector fields.
We take a coordinate system $(u,v)$ satisfying
$\xi=\partial_u,\eta=\partial_v$ at $0$.
We set an $SB$-adapted vector fields $(\tilde \xi,\tilde \eta)$
by \eqref{eq:sbadvf} with \eqref{eq:sbadvf2},
where $f_{uv}=\alpha f_u+\beta f_{vv}$ at $0$.
Since $f$ is of $S$-type, $\tilde\xi^2\tilde\eta f
=\alpha_1 \tilde\xi f+\beta_1\tilde\eta^2f$ at $0$ holds.
Setting 
$$\bar\xi=a_1\partial_u+b_1\partial_v,\quad
\bar \eta=c_1\partial_u+d_1\partial_v,$$
where 
\begin{align}
a_1&=1-\alpha v +(-\alpha_1 - 3 \alpha^2 \beta)uv,\ 
b_1=-\beta+(-3 \alpha \beta^2 - \beta_1)u,\label{eq:fors3vf}\\
c_1&=-\alpha u+(-\alpha_1-2\alpha^2 \beta)u^2/2,\ 
d_1=1,\nonumber
\end{align}
we see
$$
\bar\xi\bar\eta f=
\bar\eta\bar\xi f=
\bar\xi^2\bar\eta f=
\bar\xi\bar\eta\bar\xi f=
\bar\eta\bar\xi^2 f=0
$$
at $0$.
This proves the assertion.
\end{proof}
We remark that in the above proof, the vector field
is not made from an $SB$-$2$-adapted vector fields $(\xi,\eta)$
but from an adapted vector fields $(\partial_u,\partial_v)$.
\begin{definition}
An $SB$-$2$-adapted coordinate system $(u,v)$ is said to be
$S$-$3$-{\it adapted\/} if
$f_{uuv}(0)=0$ holds.
An $SB$-$2$-adapted vector fields $(\xi,\eta)$ is said to be
$S$-$3$-{\it adapted\/} if
$\eta\xi^2f(0)=\xi\eta\xi f(0)=\xi^2\eta f(0)=0$ holds.
\end{definition}
\begin{lemma}
The existence of an\/ $S$-$3$-adapted vector fields\/ $(\xi,\eta)$
implies the existence of an\/ $S$-$3$-adapted coordinate system\/ $(u,v)$.
\end{lemma}
\begin{proof}
By definition, if there exists an $S$-$3$-adapted vector fields,
then $f$ is of $S$-type.
Thus there exists an $SB$-$2$-adapted coordinate system.
The condition that $f$ is of $S$-type does not depend on the choice
of $SB$-$2$-adapted vector fields,
by Lemma \ref{lem:s3adapted}, there exists an
$S$-$3$-adapted coordinate system.
\end{proof}

\begin{lemma}
Let\/ $f$ be of\/ $S$-type.
$(1)$ Let\/ $(u,v)$ be an\/ $S$-$3$-adapted coordinate system.
An\/ $SB$-$2$ adapted coordinate system\/ $(x,y)=(x(u,v),y(u,v))$ 
$($i.e., $x_v=0, x_{uv}=0,y_u=0$ hold\/$)$
is of\/ $S$-$3$-adapted if and only if\/
$x_{uuv}=0,x_{uu}=0$.
$(2)$ Let\/ $(\xi,\eta)$ be an\/ $S$-$3$-adapted vector fields.
An\/ $SB$-$2$ adapted vector fields\/ $(\tilde\xi,\tilde\eta)$ where\/
$\tilde \xi=a\xi+b\eta$,
$\tilde \eta=c\xi+d\eta$ $($i.e., 
$c=0,\eta a=0,b=0,\xi c=0$ hold\/$)$
is of\/ $S$-$3$-adapted if and only if\/
$\xi\eta a=0,\xi b=0,\xi^2c=0$.
\end{lemma}
\begin{proof}
One can see this lemma by a direct calculation.
\end{proof}
\begin{lemma}\label{lem:s2cond}
Let\/ $f$ be of\/ $S$-type.
Then for an\/ $S$-$3$-adapted coordinate system\/ $(u,v)$,
the condition\/ $(\phi(\partial_u,\partial_v))_{uuu}(0)\ne0$ is equivalent to\/
$\det(f_u,f_{uuuv},f_{vv})(0)\ne0$, and it
does not depend on the choice of\/
$S$-$3$-adapted coordinate system and diffeomorphism on the target space.
For an\/ $S$-$3$-adapted vector fields\/ $(\xi,\eta)$,
the condition\/
$\xi^3\phi(\xi,\eta)\ne0$ is equivalent to\/
$\det(\xi f,\xi^3\eta f,\eta^2f)\ne0$,
and it
does not depend on the choice of\/
$S$-$3$-adapted vector fields and diffeomorphism on the target space.
\end{lemma}
\begin{proof}
It is enough to show the case of vector fields.
Noticing $\xi^2\eta f=0$, we have
$\xi^3\phi(\xi,\eta)=\det(\xi f,\xi^3\eta f,\eta^2f)$ at $0$.
Taking an $S$-$3$-adapted coordinate system $(u,v)$, we set
$
\xi=a\partial_u+b\partial_v,\quad
\eta=c\partial_u+d\partial_v.
$
Then we have
$$
\xi^3\eta f
=
*f_u+a^3df_{uuuv}+*f_{vv}.
$$
This proves the assertion.
For the independence of target diffeomorphisms, 
see Lemma \ref{lem:target}.
\end{proof}

\subsection{The condition for\/ $B_2$ singularity}
\begin{lemma}\label{lem:vfexist02}
Let us assume\/ $f$ is of\/ $B$-type.
$(1)$ There exists an\/ $SB$-$2$-adapted coordinate system\/ $(u,v)$
such that\/ $f_{vvv}(0)=0$.
$(2)$ There exists an\/ $SB$-$2$-adapted vector fields system\/ $(\xi,\eta)$
such that\/ $\eta^3f(0)=0$.
\end{lemma}
\begin{proof}
(1)
By the assumption, there exists an $SB$-$2$-adapted coordinate system $(x,y)$.
Since $f$ is of $B$-type, 
$\det(f_x, f_{yy},f_{yyy})(0)=0$ holds, we set
$f_{yyy}=\alpha f_x+\beta f_{yy}$ at $0$.
Setting a coordinate system $(u,v)$ by
$x=u - \alpha v^3/6, y= v - \beta v^2/6$
then by Lemma \ref{lem:2adaptedcond}, $(u,v)$ is of $SB$-$2$-adapted,
and differentiating
$f(u - \alpha v^3/6, v - \beta v^2/6)$, we see the $(u,v)$ is a desired
coordinate system.
(2)
Let $(\xi,\eta)$ be an adapted vector fields.
We take a coordinate system $(u,v)$ satisfying
$\xi=\partial_u,\eta=\partial_v$ at $0$.
We set an $SB$-adapted vector fields $(\tilde \xi,\tilde \eta)$
by \eqref{eq:sbadvf} with \eqref{eq:sbadvf2},
where $f_{uv}=\alpha f_u+\beta f_{vv}$ at $0$.
Since $f$ is of $B$-type, $\tilde\eta^3 f
=\alpha_1 \tilde\xi f+\beta_1\tilde\eta^2f$ at $0$ holds.
Setting 
$$\bar\xi=a_1\partial_u+b_1\partial_v,\quad
\bar \eta=c_1\partial_u+d_1\partial_v,$$
where 
$$
a_1=1-\alpha v,\ 
b_1=-\beta,\ 
c_1=-\alpha u-\alpha_1v^2/2,\ 
d_1=1-\beta_1v/3,
$$
we see
$$
\bar\xi\bar\eta f=
\bar\eta\bar\xi f=
\bar\eta^3 f=
0
$$
at $0$.
This proves the assertion.
\end{proof}
\begin{definition}
An $SB$-$2$-adapted coordinate system $(u,v)$ is said to be
$B$-$3$-{\it adapted\/} if
$f_{vvv}(0)=0$ holds.
An $SB$-$2$-adapted vector fields $(\xi,\eta)$ is said to be
$B$-$3$-{\it adapted\/} if
$\eta^3f(0)=0$ holds.
\end{definition}
We see
the existence of $B$-$3$-adapted vector fields
implies $B$-$3$-adapted coordinate system by Lemma \ref{lem:vfexist02}.
\begin{lemma}
Let\/ $f$ be of\/ $S$-type.
$(1)$ Let\/ $(u,v)$ be an\/ $B$-$3$-adapted coordinate system.
An\/ $SB$-$2$ adapted coordinate system\/ $(x,y)=(x(u,v),y(u,v))$ 
$($i.e., $x_v=0, x_{uv}=0,y_u=0$ hold\/$)$
is of\/ $B$-$3$-adapted if and only if
\begin{equation}\label{eq:2adaptedcondcoord}
x_{vvv}=0,\ y_{vv}=0
\end{equation}
$(2)$ Let\/ $(\xi,\eta)$ be a\/ $B$-$3$-adapted vector fields.
An\/ $SB$-$2$ adapted vector fields\/ $(\tilde\xi,\tilde\eta)$ where\/
$\tilde \xi=a\xi+b\eta$,
$\tilde \eta=c\xi+d\eta$ $($i.e., 
$c=0,\eta a=0,b=0,\xi c=0$ hold\/$)$
is of\/ $S$-$3$-adapted if and only if
\begin{equation}\label{eq:2adaptedcond}
\eta^2c=0,\ \eta d=0.
\end{equation}
\end{lemma}
\begin{proof}
(1) Differentiating $f(x(u,v),y(u,v))$, we see the assertion.
(2)
Calculating $\tilde \eta^3f$, then we have
$$
\tilde \eta^3f
=
(\eta c\eta d+d\eta^2 f)\xi f
+
2d\eta d\eta^2f.
$$
This shows the assertion.
\end{proof}
\begin{lemma}\label{lem:btypeindep}
Let\/ $f$ be of\/ $B$-type.
$(1)$ For a\/ $B$-$3$-adapted coordinate system\/ $(u,v)$,
the condition
\begin{equation}\label{eq:b2cricoord}
-5\det(f_u,f_{vv},f_{uvvv})^2
+
3\det(f_u,f_{vv},f_{uuv})
 \det(f_u,f_{vv},f_{vvvvv})\ne0
\end{equation}
does not depend on the choice of\/
$B$-$3$-adapted vector fields and diffeomorphism on the target space.
$(2)$ For a\/ $B$-$3$-adapted vector fields\/ $(\xi,\eta)$,
the condition
\begin{equation}\label{eq:b2cri}
-5\det(\xi f,\eta^2f,\xi\eta^3f)^2
+
3\det(\xi f,\eta^2f,\xi^2\eta f)
 \det(\xi f,\eta^2f,\eta^5f)\ne0
\end{equation}
does not depend on the choice of\/
$S$-$3$-adapted vector fields and diffeomorphism on the target space.
Moreover, in the above determinants, one can use any of\/
$\eta\xi\eta^2f$, $\eta^2\xi\eta f$ or\/ $\eta^3\xi f$ instead of\/ $\xi\eta^3f$,
and also one can use any of\/
$\eta\xi^2 f$ or\/ $\xi\eta\xi\eta f$ instead of\/ $\xi^2\eta f$.
\end{lemma}
\begin{proof}
It is enough to show the case of vector fields.
By the existence of $B$-$3$-adapted vector fields,
there exists a $B$-$3$-adapted coordinate system $(u,v)$.
We set
$$
\xi=a\partial_u+b\partial_v,\quad
\eta=c\partial_u+d\partial_v,$$
where
$$
a_v=0,\ b=0,\ c_u=0,\ 
c_{vv}=0,\ d_v=0.
$$
Then $(\xi,\eta)$ is of $B$-$3$-adapted.
Then we have
\begin{align*}
&\xi^2\eta f=a^2ddf_{uuv}+*f_u+*f_{vv},\\
&\xi\eta^3f=
ad^3(df_{uvvv}+3c_vf_{uuv})+*f_u+*f_{vv},
\end{align*}
and
\begin{align*}
&\xi^2\eta f-\xi\eta\xi f,\quad
\xi\eta\xi f-\eta^2\xi f\in\langle f_u,f_{vv}\rangle_{\R},\\
&\xi\eta^3f-\eta\xi\eta^2f,\quad
\eta\xi\eta^2f-\eta^2\xi\eta f,\quad
\eta^2\xi\eta f-\eta^3\xi f\in\langle f_u,f_{vv}\rangle_{\R}
\end{align*}
at $0$. Thus we see
$$
\eta^5f
=
d^3(d^2f_{vvvvv}+10 d c_v f_{uvvv} + 15 c_v^2 f_{uuv})+*f_u+*f_{vv}.
$$
Substituting these formula into
\eqref{eq:b2cri}, the left hand side of \eqref{eq:b2cri} is
$$
a^2d^6\big(-5\det(f_u,f_{vv},f_{uvvv})^2+3\det(f_u,f_{vv},f_{vvvvv})
\det(f_u,f_{vv},f_{uuv})\big).
$$
This shows the assertion.
For the independence of target diffeomorphisms, 
see Lemma \ref{lem:target}.
\end{proof}

\subsection{The condition for\/ $H_2$ singularity}
Let $f$ be of $HP$-type, namely, we assume for an 
adapted coordinate system $(u,v)$, it holds that
$f_u\times f_{vv}=0$, $f_u\times f_{uv}\ne0$ at $0$,
or an adapted vector fields $(\xi,\eta)$, it holds that
$\xi f\times \eta^2 f=0$,
$\xi f\times \xi\eta f\ne0$ at $0$.

\begin{lemma}
Let us assume\/ $f$ is of\/ $HP$-type.
$(1)$ There exists an adapted coordinate system\/ $(u,v)$ such that\/
$f_{vv}(0)=0$.
$(2)$ There exists an adapted vector fields\/ $(\xi,\eta)$ such that\/
$\eta^2f(0)=0$.
\end{lemma}
\begin{proof}
Let $(x,y)$ be an adapted coordinate system.
(1) By the assumption, we set $f_{yy}=\alpha f_x$ at $0$.
Setting
$x=u-\alpha v^2/2$, $y=v$ and
differentiating
$f(u-\alpha v^2/2,v)$ twice by $v$, we have
$
f_{yy}-\alpha f_x.
$
This shows the assertion.
$(2)$
By the assumption, we set 
$f_{yy}=\alpha f_x$ at $0$.
Setting 
$\xi=\partial_x$,
\begin{equation}\label{eq:h2eta}
\eta=-\alpha y \partial_x+\partial_y,
\end{equation} we see
$\eta^2f=0$ at $0$.
\end{proof}
\begin{definition}
An adapted coordinate system $(u,v)$ is of $H$-$2$-{\it adapted\/}
if $f_{vv}(0)=0$.
An adapted vector fields $(\xi,\eta)$ is of $H$-$2$-{\it adapted\/}
if $\eta^2f(0)=0$.
\end{definition}
Since the condition for an $H$-$2$-adapted coordinate system
does not depend on the non-zero functional multiplicity,
the existence of an
$H$-$2$-adapted vector fields implies the existence of
an $H$-$2$-adapted coordinate system.

\begin{lemma}\label{lem:h22adaptedcond}
Let\/ $f$ be of\/ $HP$-type.
$(1)$ Let\/ $(u,v)$ be an\/ $H$-$2$-adapted coordinate system.
An adapted coordinate system\/ $(x,y)=(x(u,v),y(u,v))$ is of\/ $H$-$2$-adapted
if and only if\/ $x_{vv}(0)=0$.
$(2)$ 
Let\/ $(\xi,\eta)$ be an\/ $H$-$2$-adapted vector fields.
An adapted vector fields\/
$\tilde \xi=a\xi+b\eta$,
$\tilde \eta=c\xi+d\eta$ $($namely, $c(0)=0)$
is of\/ $H$-$2$-adapted if and only if\/ $\eta c(0)=0$.
\end{lemma}
\begin{proof}
By a direct calculation, one can see the assertion.
\end{proof}
\begin{lemma}
Let\/ $f$ be of\/ $HP$-type.
$(1)$ For a\/ $H$-$2$-adapted coordinate system\/ $(u,v)$,
the condition\/
$\det(f_u,f_{uv},f_{vvv})(0)\ne0$ does not depend on the choice of\/
$H$-$2$-adapted vector fields and diffeomorphism on the target space.
$(2)$ For an\/ $H$-$2$-adapted vector fields\/ $(\xi,\eta)$,
the condition\/
$\det(\xi f,\xi\eta f,\eta^3f)(0)=0$
does not depend on the choice of\/
$H$-$2$-adapted vector fields and diffeomorphism on the target space.
\end{lemma}
\begin{proof}
It is enough to show the case of vector fields.
By the existence of $H$-$2$-adapted vector fields,
there exists an $H$-$2$-adapted coordinate system $(u,v)$.
We set $\xi=a\partial_u+b\partial_v$,
$\eta=c\partial_u+d\partial_v$ where $c=\eta c=0$ at $0$
Then $(\xi,\eta)$ is an $H$-$2$-adapted vector fields
by Lemma \ref{lem:h22adaptedcond}.
By a direct calculation, we see
$\eta^3f(0)=d^3f_{vvv}(0)$ and this shows the assertion.
For the independence of target diffeomorphisms, 
see Lemma \ref{lem:target}.
\end{proof}
\begin{definition}
Let $f$ be an $HP$-type map-germ.
The germ $f$ is of $H$-{\it type\/}
if $\det(\xi f,\xi\eta f,\eta^3f)(0)\ne0$ holds
for an $H$-$2$-adapted vector fields.
\end{definition}
\begin{lemma}\label{lem:h4vfexist}
Let us assume\/ $f$ is of\/ $H$-type.
$(1)$ There exists an\/ $H$-$2$-adapted coordinate system\/ $(u,v)$
such that\/ $f_{vvvv}(0)=0$.
(2) 
There exists an\/ $H$-$2$-adapted vector fields\/ $(\xi,\eta)$
such that\/ $\eta^4f(0)=0$.
\end{lemma}
\begin{proof}
(1)
By the assumption, there exists an $H$-$2$-adapted coordinate system $(x,y)$.
Since $f$ is of $H$-type, 
$f_{yyyy}=\alpha f_x+\beta f_{xy}+\delta f_{yyy}$ holds at $0$.
Setting a coordinate system $(u,v)$ by
$$
x=u - \alpha \dfrac{v^4}{4!} - \beta \dfrac{v^3}{4!},\quad
y=v - \delta \dfrac{v^2}{12},
$$
a direct calculation shows $f_{vvvv}(0)=0$.
(2)
Let $(\xi,\eta)$ be an adapted vector fields.
We take a coordinate system $(u,v)$ satisfying
$\xi=\partial_u,\eta=\partial_v$ at $0$.
We set an $H$-$2$-adapted vector fields $(\tilde \xi,\tilde \eta)$
by $\tilde\xi=\xi$ and \eqref{eq:h2eta}
where $f_{vv}=\alpha f_u$ at $0$.
Since $f$ is of $HP$-type, $\tilde\eta^4 f
=\alpha_1 \tilde\xi f+\beta_1\tilde\xi\tilde\eta f+\delta_1\tilde\eta^3f$ at $0$ holds.
Setting 
$\bar\xi=\partial_u$ and
$$\bar \eta=c_1\partial_u+d_1\partial_v,$$
where 
$$
c_1=-\alpha v+\dfrac{1}{24}(-3\beta_1-5\alpha\delta_1)v^2+
\dfrac{1}{6}(-\alpha_1-\dfrac{1}{8} \delta_1 (\beta_1+3 \alpha \delta_1))v^3,
\quad
d_1=1-\dfrac{\delta_1v}{6},
$$
we see
$$
\bar\eta^2 f=
\bar\eta^4 f=0
$$
at $0$.
This proves the assertion.
\end{proof}
\begin{definition}
An $H$-$2$-adapted coordinate system $(u,v)$ is said to be
$H$-$4$-{\it adapted\/} if
$f_{vvvv}(0)=0$ holds.
An $H$-$2$-adapted vector fields $(\xi,\eta)$ is said to be
$H$-$4$-{\it adapted\/} if
$\eta^4f(0)=0$ holds.
\end{definition}
If there exists an $H$-$2$-adapted vector fields $(\xi,\eta)$,
then there exists an $H$-$2$-adapted coordinate system $(u,v)$.
Since the condition $\det(\xi f,\xi\eta f,\eta^3f)(0)\ne0$
does not depend on the choice of $H$-$2$-adapted vector fields,
and $(\partial_u,\partial_v)$ is an $H$-$2$-adapted vector fields,
it holds that $\det(f_u,$ $f_{uv},$ $f_{vvv})(0)\ne0$.
Then there exists an $H$-$4$-adapted coordinate system
by Lemma \ref{lem:h4vfexist}.
Thus the existence of an $H$-$4$-adapted vector fields
implies the existence of an $H$-$4$-adapted coordinate system.

\begin{lemma}
Let\/ $f$ be of\/ $H$-type.
$(1)$ Let\/ $(u,v)$ be an\/ $H$-$4$-adapted coordinate system.
An\/ $H$-$2$ adapted coordinate system\/ $(x,y)=(x(u,v),y(u,v))$ 
$($i.e., $x_v(0)=x_{vv}(0)=0$ holds\/$)$
is of\/ $H$-$4$-adapted if and only if\/
$x_{vvvv}=0$, $x_{vvv}=0$, $y_{vv}=0$ at\/ $0$.
$(2)$ Let\/ $(\xi,\eta)$ be an\/ $H$-$4$-adapted vector fields.
An\/ $H$-$2$ adapted vector fields\/ $(\tilde\xi,\tilde\eta)$ where\/
$\tilde \xi=a\xi+b\eta$,
$\tilde \eta=c\xi+d\eta$ $($i.e., $c(0)=\eta c(0)=0$ holds\/$)$
is of\/ $H$-$4$-adapted if and only if\/
$\eta^2c=0$, $\eta^3 c=0$, $\eta d=0$ at\/ $0$.
\end{lemma}
\begin{proof}
It is enough to show the case of vector fields.
By a direct calculation, we have $\tilde \eta f=c\,\xi f+d\, \eta f$ and
$\tilde \eta^2 f=c\,\xi A+d\, \eta B$, where
$$
A=\xi c\,\xi f+c\, \xi^2f+\xi d\,\eta f+d\,\xi \eta f,\quad
B=\eta c\,\xi f+c\, \eta\xi f+\eta d\,\eta f+d\, \eta^2 f,
$$
and we have
\begin{align*}
\tilde\eta^3f&=
c\Big(
\xi c \,A+c\,\xi A+\xi d\, B+d\,\xi B\Big)
+
d\Big(
\eta c\, A+c\,\eta A+\eta d\, B+d\,\eta B\Big),\\
\tilde\eta^4f&=c(*)+
d\eta(\tilde\eta^3f)\\
&=
c(*)+d\Bigg(
\eta c\Big(\xi c\, A+c\,\xi A+\xi d\, B+d\,\xi B\Big)
+c(*)\\
&\hspace{10mm}
+\eta d\Big(\eta c\, A+c\,\eta A+\eta d\, B+d\,\eta B\Big)\\
&\hspace{10mm}
+
d\Big(
\eta^2 c\,A+2\eta c\,\eta A+c\,\eta^2 A
+\eta^2d\,B+2\eta d\,\eta B+d\,\eta^2B
\Big)\Bigg).
\end{align*}
Since we see that
\begin{align*}
\eta B&=
\eta^2 c\,\xi f+2\eta c\,\eta\xi f
+c\,\eta^2\xi f
+
\eta^2 d\,\eta f+2\eta d\,\eta^2 f
+d\,\eta^3f,\\
\eta^2 B&=
\eta^3 c\,\xi f+3\eta^2 c\,\eta\xi f+3\eta c\,\eta^2\xi f
+c\,\eta^3\xi f
+
\eta^3 d\,\eta f+3\eta^3 d\,\eta^2 f+3\eta d\,\eta^3 f
+d\,\eta^4f,
\end{align*}
$H$-$2$-adaptivity of $(\xi,\eta)$, and that
$f$ is of $H$-type,
we have $B(0)=0$, $\eta B(0)=\eta^2 c\,\xi f+2\eta d\,\eta^2 f$
and
$$
\eta^2 B(0)=
\eta^3 c\,\xi f+3\eta^2 c\,\eta\xi f
+3\eta d\,\eta^3 f+d\,\eta^4f.
$$
Since $f$ is of $H$-type and $(\xi,\eta)$ is of $H$-$2$-adapted, it holds that
$\eta^4f(0)=0$, and
$\{\xi f,\xi\eta f,\eta^3f\}$ is linearly independent at $0$.
Moreover, $\xi\eta f=\eta\xi f+[\xi,\eta]f$ holds, and
$[\xi,\eta]f$ is parallel to $\xi f$ at $0$,
the subspace generated by $\{\xi f,\xi\eta f,\eta\xi f\}$ is the same as
that of $\{\xi f,\xi\eta f\}$ at $0$.
Thus looking at the coefficients of the vector $\eta^3f(0)$
of $\tilde\eta^4f(0)$, we have $\eta d(0)=0$.
Then $\tilde\eta^4f(0)$ with the condition $\eta d(0)=0$ divided by $d^2$ is
$$
\eta^2c(\xi c+d \xi \eta f)+d\eta^3 c\, \xi f+4 \eta^2 c\,\eta\xi f
=
\eta^2c(\xi c+d\, \xi \eta f)+d(\eta^3 c\, \xi f
+4 \eta^2 c(\xi\eta f+[\xi,\eta]f)).
$$
Looking at the coefficients of the vector $\xi\eta f(0)$,
noticing that $[\xi,\eta]f(0)$ is parallel to $\xi f(0)$ again,
we have $\eta^2c=0$.
Then $\tilde\eta^4f(0)$ with the condition $\eta d(0)=\eta^2c(0)=0$ divided by $d^3$ is
$\eta^3c(0)\,\xi f(0)$.
We have the assertion.
\end{proof}

\begin{lemma}
Let\/ $f$ be of\/ $H$-type.
Then for an\/ $H$-$4$-adapted coordinate system\/ $(u,v)$,
the condition\/ $\det(f_u,f_{vvvvv},f_{vvv})(0)\ne0$ 
does not depend on the choice of\/
$H$-$4$-adapted coordinate system and diffeomorphism on the target space.
For an\/ $H$-$4$-adapted vector fields\/ $(\xi,\eta)$,
the condition\/
$\det(\xi f,\eta^5 f,\eta^3f)(0)\ne0$ does not depend on the choice of\/
$H$-$4$-adapted vector fields and diffeomorphism on the target space.
\end{lemma}
\begin{proof}
It is enough to show the case of vector fields.
Taking an $H$-$4$-adapted coordinate system $(u,v)$, we set
$\xi=a\partial_u+b\partial_v$,
$\eta=c\partial_u+d\partial_v$.
Then since $c=c_v=c_{vv}=c_{vvv}=d_v=0$ at $0$, we have
$$
\eta^5 f(0)=
d^4(
c_{vvvv}f_u+10 d_{vv}f_{vvv}+df_{vvvvv})(0),
$$
and this shows the assertion.
For the independence of target diffeomorphisms, 
see Lemma \ref{lem:target}.
\end{proof}

Here, we give a proof of the ``only if part'' of Theorem \ref{thm:cri}
by the independence of the above conditions.
\begin{proof}
[Proof of Theorem\/ {\rm \ref{thm:cri}}, ``only if part'']
\ref{itm:cri1} We assume $f$ is an $S_2$ singularity.
Then there exist diffeomorphism-germs 
if there exist diffeomorphism-germ
$\phi_s:(\R^2,0)\to(\R^2,0)$ and $\phi_t:(\R^3,0)\to(\R^3,0)$
such that
$\phi_t\circ f\circ \phi_s^{-1}=(u,v^2,v(u^3+v^2))$ holds.
This coordinate $(u,v)$ is of $S$-$3$-adapted, and
since the condition \eqref{eq:cri1} does not depend on the choice of
coordinate system, one can calculate \eqref{eq:cri1} by $(u,v)$.
Then one can easily see that condition is satisfied.
This shows the assertion.
\ref{itm:cri2} We assume $f$ is a $B_2^\pm$ singularity.
Then there exist diffeomorphism-germs 
if there exist diffeomorphism-germ
$\phi_s:(\R^2,0)\to(\R^2,0)$ and $\phi_t:(\R^3,0)\to(\R^3,0)$
such that
$\phi_t\circ f\circ \phi_s^{-1}=(u,v^2,v(u^2\pm v^4))$ holds.
This coordinate $(u,v)$ is of $B$-$3$-adapted, and
since the condition \eqref{eq:cri2} does not depend on the choice of
coordinate system, one can calculate \eqref{eq:cri2} by $(u,v)$.
Then one can easily see that condition is satisfied.
This shows the assertion.
\ref{itm:cri3} We assume $f$ is an $H_2$ singularity.
Then there exist diffeomorphism-germs 
if there exist diffeomorphism-germ
$\phi_s:(\R^2,0)\to(\R^2,0)$ and $\phi_t:(\R^3,0)\to(\R^3,0)$
such that
$\phi_t\circ f\circ \phi_s^{-1}=(u,uv+v^5,v^3)$ holds.
This coordinate $(u,v)$ is of $H$-$4$-adapted, and
since the condition \eqref{eq:cri3} does not depend on the choice of
coordinate system, one can calculate \eqref{eq:cri3} by $(u,v)$.
Then one can easily see that condition is satisfied.
This shows the assertion.
\end{proof}

\section{Conditions in terms of coefficients of jets}
\subsection{Conditions for\/ $S_2$ and\/ $B_2$ singularities}
We set $f=\Big(u,v^2/2+b(v),a(u,v)\Big)$, where
\begin{align}
\label{eq:normalsba}
a(u,v)&=\sum_{i+j=3}^5 \dfrac{a_{ij}}{i!j!}u^iv^j\quad
(a_{30}=a_{40}=a_{50}=0)\\
\label{eq:normalsbb}
b(v)&=\sum_{j=3}^5\dfrac{b_{0i}}{i!}v^i.
\end{align}
\begin{lemma}{\rm \cite[Proposition 2.3]{fhskbk}}\label{lem:jetcondskbk}
The above map-germ\/ $f$ at\/ $0$ is an\/ $S_2$ singularity if and only if
$$
a_{21}=0,\quad
a_{31}\ne0,\quad
a_{03}\ne0.
$$
The above map-germ\/ $f$ at\/ $0$ is a\/ $B_2^\pm$ singularities if and only if
$$
a_{03}=0,\quad
a_{21}\ne0,\quad
3a_{05}a_{21}-5a_{13}^2\ne0.
$$
Moreover, the\/ $\pm$-sign of the\/ $B_2^\pm$ singularity 
is determined by the sign of\/ $3a_{05}a_{21}-5a_{13}^2$.
\end{lemma}
The proof in \cite[Proposition 2.3]{fhskbk} does not contain the
coincidence of the sign of the $B_2^\pm$ singularities, however,
following their argument, one can see the assertion.
\subsection{Condition for\/ $H_2$ singularity}
We set
\begin{equation}\label{eq:start}
f=\Big(u,u v+a(u,v),b(u,v)\Big),
\end{equation}
where 
\begin{equation}\label{eq:start2}
a(u,v)=\sum_{i+j=3}^5 \dfrac{a_{ij}}{i!j!}u^iv^j,\quad
b(u,v)=
\sum_{i+j=3}^5 \dfrac{b_{ij}}{i!j!}u^iv^j.
\end{equation}
\begin{lemma}\label{lem:jetcondh2}
The map-germ\/ $f$ at\/ $0$ is an\/ $H_2$ singularity if and only if
$$b_{03}\ne0,\quad c\ne0,$$
where
\begin{align}
c=&
(4 a_{05} - 10 a_{04} a_{12}) b_{03}^2 + 
 (-5 a_{04} b_{04} - 4 a_{03} b_{05} + 10 a_{03} a_{12} b_{04} 
+ 10 a_{03} a_{04} b_{12})b_{03}\nonumber  \\
&\hspace{20mm}+ 
 a_{03} (5 b_{04}^2 - 10 a_{03} b_{04} b_{12}).
\end{align}
\end{lemma}
If $a_{12} = 0, a_{03} = 0$, then the condition 
$c\ne0$ is equivalent to
\begin{equation}\label{eq:lem:jetcondh2sp}
b_{03} (4 a_{05} b_{03} - 5 a_{04} b_{04})\ne0.
\end{equation}
For a proof of this lemma, we use the following notation.
For a function
$$g=\sum_{i+j=1}^k\dfrac{g_{ij}}{i!j!}u^iv^j,$$
we set
$$
\coef(g,u,v,k)=
\Big(
(g_{10},g_{01}),
(g_{20},g_{11},g_{02}),
(g_{30},g_{21},g_{12},g_{03}),
\cdots,
(\cdots,g_{0k})\Big).
$$
Furthermore, we set
\begin{align*}
\coef\Big((g_1,g_2,g_3),u,v,k\Big)
&=
\pmt{
\coef(g_1,u,v,k)\\
\coef(g_2,u,v,k)\\
\coef(g_3,u,v,k)}.
\end{align*}
In \cite[Lemma 4.2.1:1]{mond}, it is shown that there 
is only one $\A^4$-orbit over $(u,uv,v^3)$.
This implies that the term $av^4$ in $(u,uv,v^3+av^4)$
can be eliminated by an $\A$-action.
The Mather lemma is used in its proof.
The proof does not produce a specific coordinate changes, 
and the conditions are not obtained.
Our proof given in here produces concrete coordinate changes.
\begin{proof}
We remark that $H_2$ is $5$-determined (\cite[Theorem 4.2.1:2]{mond}).
We assume $b_{03}\ne0$.
By the coordinate change $v\mapsto v+(b_{12}/b_{03})u$ and
by a diffeomorphism on the target space,
eliminating the terms of $u^i$ $(i=3,4,5)$ from the second and the third
components, and the terms of $u^2v$ from the third component, we see
$$
\coef(f,u,v,3)
=
\pmt{
(1,0),(0,0,0),(0,0,0,0)\hfill\\
(0,0),(0,1,0),(0,a_{03} b_{12}^2/(2 b_{03}^2),-a_{03} b_{12}/(2 b_{03}),a_{03}/6)\\
(0,0),(0,0,0),(0,0,0,b_{03}/6)\hfill}.
$$
By the coordinate change $v\mapsto v+a_{03} b_{12}/(2 b_{03})v^2$ and
by a diffeomorphism on the target space,
eliminating the terms of $u^3,u^2v$ from the second component,
eliminating the terms of $v^3$ from the third component,
and dividing the third component by $b_{03}/6$, we have
$$
\coef(f,u,v,3)
=
\pmt{
(1,0),(0,0,0),(0,0,0,0)\hfill\\
(0,0),(0,1,0),(0,0,0,0)\hfill\\
(0,0),(0,0,0),(0,0,0,1/6)}.
$$
By a diffeomorphism on the target space, we have
\begin{align*}
&\coef(f,u,v,4)\\
=&
\pmt{
(1,0),(0,0,0),(0,0,0,0),(0,0,0,0,0)\hfill\\
(0,0),(0,1,0),(0,0,0,0),(0,0,0,0,(a_{04} b_{03}-a_{03} b_{04})/(24 b_{03}))
\hfill\\
(0,0),(0,0,0),(0,0,0,1/6),
\left(0,0,0,0,
\dfrac{-6 a_{12} b_{03}+b_{04}+6 a_{03} b_{12}}{24 b_{03}}\right)}.
\end{align*}
By taking the coordinate change
$u\mapsto u-(a_{04} b_{03}-a_{03} b_{04})/(24 b_{03}))v^3/4!$, 
the third component is unchanged, and the first and the second
components $f_1,f_2$ satisfy
\begin{align*}
&\pmt{
\coef(f_1,u,v,4)\\
\coef(f_2,u,v,4)}\\
=&
\pmt{
(1,0),(0,0,0),(0,0,0,(a_{04} b_{03}-a_{03} b_{04})/(24 b_{03}))),
(0,0,0,0,0)\\
(0,0),(0,1,0),(0,0,0,0),(0,0,0,0,0)
\hfill}.
\end{align*}
Since the term of $v^3$ of the third component is $v^3/6$,
one can eliminate the term of $v^3$ of the first component
by a diffeomorphism on the target space.
Then the second and the third components are unchanged
and the first component $f_1$ satisfies
\begin{align*}
\coef(f_1,u,v,4)
&=
(1,0),(0,0,0),(0,0,0,0),\\
&\hspace{10mm}
\left(0,0,0,0,
-\dfrac{(a_{04} b_{03}-a_{03} b_{04}) (6 a_{12} b_{03}-b_{04}-6 a_{03} b_{12})}{96 b_{03}^2}\right).
\end{align*}
By these modifications the third component is now
$v^3/6+(-6 a_{12} b_{03}+b_{04}+6 a_{03} b_{12})v^4$.
if $-6 a_{12} b_{03}+b_{04}+6 a_{03} b_{12}=0$, then the coefficient
of $v^4$ of the first component vanishes.
If $-6 a_{12} b_{03}+b_{04}+6 a_{03} b_{12}=0$, then the term
of $v^4$ of the first component can be eliminated by the target
diffeomorphism.
Let $\tilde f=(\tilde f_1,\tilde f_2,\tilde f_3)$ be the result of
the above modifications up to here.
On the other hand, we set $g$ satisfying
\begin{align}\label{eq:h2changelast}
 &\coef(g,u,v,5)\\
=&
\pmt{
(1,0),(0,0,0),(0,0,0,0),(0,0,0,0,0),
(g_{1,50},g_{1,41},g_{1,32},g_{1,23},g_{1,14},g_{1,05})\hfill\\
(0,0),(0,1,0),(0,0,0,0),(0,0,0,0,0),
(g_{2,50},g_{2,41},g_{2,32},g_{2,23},g_{2,14},g_{2,05})\hfill\\
(0,0),(0,0,0),(0,0,0,1/6),(0,0,0,0,0),
(g_{3,50},g_{3,41},g_{3,32},g_{3,23},g_{3,14},g_{3,05})}.
\nonumber
\end{align}
Then by a coordinate change on the source and target spaces,
coefficients except $g_{2,05}$ can be eliminated 
$0$ without changing other parts.
Thus $\tilde f$ is an $H_2$ singularity if and only if
the coefficient of $v^5$ in $\tilde f_2$ does not vanish.
Since the coefficient is
\begin{align*}
&\dfrac{1}{480 b_{03}^2}
\Big(
4 a_{05} b_{03}^2-5 a_{04} b_{03} (2 a_{12} b_{03}+b_{04}-2 a_{03} b_{12})\\
&\hspace{20mm}
+a_{03} (10 a_{12} b_{03} b_{04}+5 b_{04}^2-4 b_{03} b_{05}-10 a_{03} b_{04} b_{12})
\Big),
\end{align*}
this shows the assertion for the ``if part''.
The ``only if part'' can be proven by the same consideration
of the proof of Theorem {\rm \ref{thm:cri}}.
\end{proof}

\section{Proof of criteria}
\begin{proof}
[Proof of Theorem\/ {\rm \ref{thm:cri}\ref{itm:cri1}} and\/ \ref{itm:cri2}]
We assume that $f$ is of $S$-type or $B$-type and satisfies the
assumption of the theorem.
Since $f$ satisfies $\xi f\times \eta^2f(0)\ne0$,
the $2$-jet $j^2f$ is $\A$-equivalent to $(u,v^2,0)$.
Then $f$ is $\A$-equivalent to the form $(u,v^2/2,a(u,v))$ with
\eqref{eq:normalsba} and the condition $a_{12}=a_{22}=0$.
Then $(u,v)$ is of $SB$-$2$-adapted.
We assume $f$ is of $S$-type and satisfies the
assumption \ref{itm:cri1} of the theorem.
Since $f$ is of $S$-type, $a_{21}=0$, $a_{03}\ne0$ hold.
Then $(u,v)$ is of $S$-$3$-adapted.
By the assumption of the theorem, we have $a_{31}\ne0$.
By Lemma \ref{lem:jetcondskbk}, the assertion \ref{itm:cri1} is proven.
Next we assume $f$ is of $B$-type and satisfies the
assumption \ref{itm:cri3} of the theorem.
Since $f$ is of $B$-type, $a_{21}\ne0$, $a_{03}=0$ hold.
Then $(u,v)$ is of $B$-$3$-adapted.
By Lemma \ref{lem:btypeindep}, the condition \eqref{eq:b2cricoord}
can be calculated by this coordinate system.
Then the condition \eqref{eq:b2cricoord} can be calculated as
$-5a_{13}^2+3a_{21}a_{05}\ne0$.
By Lemma \ref{lem:jetcondskbk}, the assertion \ref{itm:cri2} is proven.
\end{proof}
\begin{proof}[Proof of Theorem\/ {\rm \ref{thm:cri}\ref{itm:cri3}}]
We assume that $f$ is of $S$-type or $B$-type and satisfies the
assumption \ref{itm:cri3} of the theorem.
Since $f$ satisfies $\xi f\times \eta^2f(0)=0$, $\xi f\times \xi\eta f(0)\ne0$,
the $2$-jet $j^2f$ is $\A$-equivalent to $(u,uv,0)$.
Thus we may assume that $f$ is given by the form \eqref{eq:start}
with \eqref{eq:start2}.
Then 
Then $(u,v)$ is of $H$-$2$-adapted.
Since $f$ is of $H$-type, $\det(f_u,f_{uv},f_{vvv})\ne0$ holds.
So, the coefficient of $v^3$ in $b$ does not vanish.
By a coordinate change on the targe space, one can eliminate the
term of $v^3$ in $a$.
Thus we can write the second component of $f$ as
$u(v+\tilde a_1(u,v))+a_{04}v^4+a_{05}v^5$.
By a coordinate change
$v\mapsto v+\tilde a_1(u,v)$,
the second component is changed to
$uv+a_{04}v^4/4!+a_{05}v^5/5!$.
In particular, the coefficient of $v^3$ still vanishes.
We may assume that the term of $v^3$ in $b$ is
$v^3/6$.
Let $g$ be the map with the above modifications.
Then since
$
g_u=(1,0,0),
g_{uv}=(0,1,0),
g_{vvv}=(0,0,1)
$, we see
$$g_{vvvv}=\alpha g_u+\beta g_{uv}+\delta g_{vvv}$$
holds, where
$\alpha=0$, $\beta=a_{04}$, $\delta=b_{04}$.
We set
$\xi=\partial_u$,
$$
\eta=
\left(\dfrac{a_{04}}{2}u-\dfrac{a_{04}}{8}v^2\right)\partial_u
+
\left(1-\dfrac{b_{04}}{6}v\right)\partial_v.
$$
Then $(\xi,\eta)$ is of $H$-$4$-adapted.
Noticing that the coefficient of the term $uv^2$ in the second
component of $g$ vanishes, we see
$$\eta^5f=*f_{u}+*f_{vvv}+a_{05}f_{vvvvv}-(5/4)a_{04}b_{04}f_{uv}$$
at $0$.
By the assumption of the theorem, 
$\det(f_u,f_{vvvvv},f_{vvv})\ne0$ holds.
Thus we have
$a_{05}-(5/4)a_{04}b_{04}\ne0$.
Since there are no third order terms in the second component,
$g$ is an $H_2$ singularity if and only if the condition
\eqref{eq:lem:jetcondh2sp} does not vanish in Lemma
\ref{lem:jetcondh2}.
This is equivalent to
$a_{05}-(5/4)a_{04}b_{04}\ne0$.
Thus $g$ is an $H_2$ singularity, and so is $f$.
\end{proof}

\section{Singularities of surfaces}
In this section, as an application of  our criteria, we give 
conditions for codimension two
singularities that appear in ruled surfaces and center maps of surfaces in the Euclidean space and folded surfaces
with respect to a plane.
\subsection{Singularities of ruled surfaces}
Let $\gamma:(\R,0)\to(\R^3,0)$ be a curve
and let $a_1,a_2,a_3:(\R,0)\to\R^3$ be an orthonormal
frame field along $\gamma$, namely, $|a_i|=1$ ($i=1,2,3$)
and $a_i\cdot a_j=0$ ($i\ne j$).
Let us consider a ruled surface
$$
f(u,v)=\gamma(v)+ua_1(v).
$$
If $a_1$ has a singular point at $v=v_0$, 
then $f$ is regular on the line $l:u\mapsto \gamma(v_0)+ua_1(v_0)$
or $f$ is singular on $l$.
This implies $f$ does not have $S_2$, $B_2^\pm$, $H_2$ singularities.
So we assume $a_1'\ne0$, where $'=d/dv$ in this section.
If $a_1'\ne0$, then $f$ is said to be {\it non-cylindrical\/}.
Under this assumption, 
we may assume $a_1'=a_2$ by taking a suitable parameter and
a change of $a_2,a_3$.
By a change of $u\mapsto u-\gamma'\cdot a_1'/a_1'\cdot a_1'$,
we may assume $\gamma'\cdot a_1'=0$. 
This $\gamma$ is called a {\it striction curve}.
See \cite[Section 17]{gray} for these terminology.
We define functions $\gamma_1,\gamma_2,\gamma_3$, $c_1,c_2,c_3$ by
$$
\gamma'=\gamma_1a_1+\gamma_2a_2+\gamma_3a_3,\quad
\pmt{a_1'\\ a_2'\\ a_3'}
=
\pmt{
0&c_1&c_2\\
-c_1&0&c_3\\
-c_2&-c_3&0}
\pmt{a_1\\ a_2\\ a_3}.
$$
Since we assumed $\gamma'\cdot a_1'=\gamma'\cdot a_2=0$,
and $a_1'=a_2$, 
it holds that $\gamma_2=0$, $c_1=1$, $c_2=0$ for any $v$.
We assume $0$ is a singular point of $f$.
Then $\gamma_3(0)=0$ holds.
We set $G_1$ as $G_1'=\gamma_1$ and $G_1(0)=0$.
By changing $u\mapsto u-G_1(v)$, 
we re-set
$$f(u,v)=\gamma(v)+(u-G_1(v))a_1(v).$$
\begin{theorem}
The ruled surface\/ $f$ at\/ $0$ is 
\begin{enumerate}
\renewcommand{\theenumi}{$\mathrm{\arabic{enumi}}$}
\item a Whitney umbrella if and only if\/ $\gamma_3'(0)\ne0$,
\item an\/ $S_1^\pm$ singularity if and only if\/ $\gamma_3'(0)=0$,
$\gamma_1(0)\ne0$
 and\/ $\gamma_3''(0) (2c_3(0)\gamma_1(0)+\gamma_3''(0))\ne0$. 
Moreover, the\/ $\pm$ sign is determined by the sign of\/
$\gamma_3''(0) (2c_3(0)\gamma_1(0)+\gamma_3''(0))$,
\item an\/ $S_2$ singularity if and only if\/
$\gamma_3'(0)=\gamma_3''(0)=0$,
and\/ $c_3(0)\gamma_1(0)\gamma_3'''(0)\ne0$,
 \item a\/ $B_2^\pm$ singularity if and only if\/ 
$\gamma_1(0)\ne0$, $c_3(0)=\gamma_3''(0)/2\gamma_1(0)$, $\gamma_3'(0)=0$,
$\gamma_3''(0)\ne0$  and\/
$b\ne0$, where
\begin{align*}
b=&-20 c_3'(0)^2 \gamma_1(0)^4
+ \Big(-12 c_3''(0) \gamma_3''(0)
+20 c_3'(0) \gamma_3'''(0)\Big)\gamma_1(0)^3\\
&\hspace{2mm}+
\bigg(-28 c_3'(0) \gamma_1'(0)\gamma_3''(0)
-5 \gamma_3'''(0)^2-24 \gamma_3''(0)^2
+3 \gamma_3''(0)\gamma_3''''(0)\bigg)\gamma_1(0)^2
\\
&\hspace{2mm}+2 \gamma_3''(0)  \Big(5\gamma_1'(0)\gamma_3'''(0) 
-3 \gamma_1''(0) \gamma_3''(0)\Big)\gamma_1(0)
-5\gamma_1'(0)^2\gamma_3''(0)^2 -3 \gamma_3''(0)^4.
\end{align*}
Moreover, the\/ $\pm$ sign coincides with that of\/ $-b$.
\item an\/ $H_2$ singularity if and only if\/ $h\ne0$, $\gamma_3''(0)\ne0$ and\/ $\gamma_3'(0)=\gamma_1(0)=0$, where
\begin{align*}
h=&24 c_3'(0) \gamma_3''(0)^3+ \Big(c_3(0)\gamma_3'''(0)+3 (5 c_3(0)^2+12) \gamma_1'(0)+4 \gamma_1'''(0)\Big) \gamma_3''(0)^2\\
&+\Big(-4 \gamma_1'(0)\gamma_3''''(0) -5\gamma_1''(0)\gamma_3'''(0) 
+21 c_3(0)\gamma_1'(0)\gamma_1''(0) \\
&\hspace{10mm}+24 c_3'(0)  \gamma_1'(0)^2\Big)
\gamma_3''(0)\\
&+5 \gamma_1'(0) \Big(\gamma_3'''(0)^2-4 c_3(0) \gamma_1'(0)\gamma_3'''(0) 
+3 c_3(0)^2 \gamma_1'(0)^2\Big).
\end{align*}
\end{enumerate}
\end{theorem}
\begin{proof}
We take a coordinate system on $\R^3$ by 
$\{a_1(0),a_2(0),a_3(0)\}$ to be the fundamental vectors.
Then we see 
\begin{equation}\label{eq:ruleddiff}
f_{u}(0)=(1,0,0),\quad
f_{uv}(0)=(0,1,0),\quad
f_{vv}(0)=(0,-\gamma_1(0),\gamma_3'(0)).
\end{equation}
Thus $\det(f_u,f_{vv},f_{uv})(0)\ne0$ if and only if $\gamma_3'(0)\ne0$.
Thus we see the condition for a Whitney umbrella.
We assume $\gamma_3'(0)=0$.
Under this condition, $f_u\times f_{vv}\ne0$ if and only if $\gamma_1(0)\ne0$.
We assume $\gamma_1(0)\ne0$.
We set $\xi=\partial_u-\beta\partial_v$, $\eta=\partial_v$.
Then $(\xi,\eta)$ is an $SB$-$2$-adapted vector fields. 
We see that the Hesse matrix of
$\phi=\phi(\xi,\eta)=\det(\xi f,\eta f, \eta^2 f)$
is
\begin{equation}\label{eq:ruledhess}
\hess\phi(0)=\pmt{
\gamma_3''(0)/\gamma_1(0)&
0\\
0&\gamma_1(0)(-2c_3(0)\gamma_1(0)+\gamma_3''(0))},
\end{equation}
and
$$
\det\hess\phi(0)=\gamma_3''(0) (2c_3(0)\gamma_1(0)+\gamma_3''(0)).
$$
So, if $\gamma_3''(0) (-2c_3(0)\gamma_1(0)+\gamma_3''(0))\ne0$, then $f$ at $0$ is an
$S_1^\pm$ singularities.
We assume $\gamma_3''(0) =0$ and $-2c_3(0)\gamma_1(0)+\gamma_3''(0)\ne0$,
namely $p$ is of $S$-type.
We set $(\bar\xi,\bar\eta)$ as in Lemma \ref{lem:s3adapted},
where
$$
\alpha=0,\quad \beta=-1/\gamma_1(0),\quad \alpha_1=0,\quad \beta_1=\gamma_1'(0)/\gamma_1(0)^3.
$$
Then $(\bar\xi,\bar\eta)$ is an $S$-$3$-adapted vector fields. 
Calculating $\det(\bar\xi f,\bar\xi^3\bar\eta f,\bar\eta^2 f)(0)$, we see
it is equal to
$$\dfrac{\gamma_3'''(0)}{\gamma_1(0)^2}.$$
This shows the assertion for $S_2$ singularity.
We assume $\gamma_3'(0)=0$ and $\gamma_1(0)\ne0$.
By \eqref{eq:ruledhess},
the condition that $f$ at $0$ is of
$B$-type if and only if $\gamma_3''(0)\ne0$ and 
$c_3(0)=\gamma_3''(0)/2\gamma_1(0)$.
We assume this condition.
Moreover, we set $(\bar\xi,\bar\eta)$ as in Lemma \ref{lem:vfexist02},
where
$$
\alpha=0,\quad \beta=-1/\gamma_1(0),\quad \alpha_1=2\gamma_1(0),\quad \beta_1=\gamma_1'(0)/\gamma_1(0).
$$
Then $(\bar\xi,\bar\eta)$ is an $B$-$3$-adapted vector fields. 
By a direct calculation, we see
$$-5\det(\bar\xi f,\bar\eta^2 f,\bar\eta^3\bar\xi f)(0)^2
+
3\det(\bar\xi f,\bar\eta^2 f,\bar\eta\bar\xi^2f)(0)
 \det(\bar\xi f,\bar\eta^2f,\bar\eta^5f)(0)=\dfrac{b}{\gamma_1(0)^2}.$$
This shows the assertion for $B_2$ singularity.
By \eqref{eq:ruleddiff},
the condition of $HP$-type if and only if $\gamma_3'(0)=\gamma_1(0)=0$,
and we assume this condition.
Then $(u,v)$ is an
$H$-$2$-adapted coordinate system, and we see
we see
$\det(f_u,f_{uv},f_{vvv})(0)=\gamma_3''(0)$. Thus we assume $\gamma_3''(0)\ne0$.
Moreover, we set $(\bar\xi,\bar\eta)$ as in Lemma \ref{lem:h4vfexist},
where
\begin{align*}
\alpha=&0,\\
\alpha_1=&3\gamma_1'(0),\\
\beta_1=&-3c_3(0)\gamma_3''(0)-\gamma_1''(0)+\dfrac{\gamma_1'(0)(-3c_3(0)\gamma_1'(0)+\gamma_3'''(0))}{\gamma_3''(0)},\\
\delta_1=&\dfrac{-3c_3(0)\gamma_1'(0)+\gamma_3'''(0)}{\gamma_3''(0)}.
\end{align*}
Then $(\bar\xi,\bar\eta)$ is an $H$-$4$-adapted vector fields.
By a direct calculation, we have
$$\det(\bar\xi f,\bar\eta^5f,\bar\eta^3f)(0)=-\dfrac{h}{4\gamma_3''(0)}.$$
This shows the assertion for $H_2$ singularity.
\end{proof}

\subsection{Singularities of Euclidean center maps}
Let $f:(\R^2,0)\to(\R^3,0)$ be a frontal,
and let $\nu$ be a unit normal of $f$.
We set $\rho:(\R^2,0)\to\R$ by
$\rho=f\cdot \nu$.
The map $c$
$$
c=f-\rho \nu
$$
is called the {\it Euclidean center map}, or 
the {\it center map}.
In this section, we study singularities of codimension
up to $2$ of center map
by using our criteria.
Generally speaking, the center map is defined in the
context of affine geometry by using a transversal vector
field instead of the Euclidean normal vector.
See \cite{furuvra} for detail.
Here we stick our considerations to the Euclidean normal case
to study its singular point.
We set
$f=(u,v,k+a(u,v))$, where $k\in\R$ and
\begin{align}
\label{eq:mongecent}
a(u,v)&=\sum_{i+j=2}^6 \dfrac{a_{ij}}{i!j!}u^iv^j,
\end{align}
$a_{11}=0$.
Then $j^1c(0,0)=((1+a_{20} k)u,(1+a_{02} k)v,0)$.
So we assume $a_{02}\ne0$, $k=-1/a_{02}$ and $a_{20}\ne a_{02}$, since
we are interested in singularities which
are dealt with this paper.
We have the following theorem.
\begin{theorem}
For the center map\/ $c$ of\/ $f$ at\/ $0$,
Whitney umbrella, $B_2^\pm$ and\/ $H_2$ singularities never appear.
The map\/ $c$ at\/ $0$ is 
\begin{enumerate}
\renewcommand{\theenumi}{{\rm \arabic{enumi}}}
\item an\/ $S_1^\pm$ singularity if and only if\/
$a_{03}(-a_{12}^2+a_{03} a_{21})\ne0$. 
Moreover, the\/ $\pm$ sign is determined by the sign of\/
$-a_{12}^2+a_{03} a_{21}$,
\item an\/ $S_2$ singularity if and only if\/
$a_{03}\ne0$, $a_{21}=a_{12}^2/a_{03}$ and\/
$s\ne0$, where
$$
s=3 a_{02}^3 a_{12}^3 - a_{04} a_{12}^3 + 
 3 a_{12}^2 a_{13} a_{03} + (3 a_{02} a_{12} a_{20}^2 - 3 a_{12} a_{22}) 
a_{03}^2 + a_{31} a_{03}^3.$$
\end{enumerate}
\end{theorem}
\begin{proof}
We see
$$
c_{uv}(0)=\left(-\dfrac{a_{21}}{a_{02}},-\dfrac{a_{12}}{a_{02}},0\right),\quad
c_{vv}(0)=\left(-\dfrac{a_{12}}{a_{02}},-\dfrac{a_{03}}{a_{02}},0\right).
$$
Thus 
$\det(c_u,c_{vv},c_{uv})(0)=0$ for any $f$.
On the other hand,
$c_u\times c_{vv}\ne0$ if and only if
$a_{03}\ne0$,
 and
 $c_u\times c_{vv}=0$, $c_u\times c_{uv}\ne0$ if and only if
$a_{03}=0, c_{12}\ne0$.
Firstly, we assume $a_{03}\ne0$.
We set $\phi=\phi(\partial_u,\partial_v)=\det(c_u,c_v,c_{vv})$.
Then we see
$$
\hess\phi(0)=\pmt{
((a_{02}-a_{20}) (2 a_{12}^2+a_{03} a_{21}))/a_{02}^2&
(3 a_{03} a_{12} (a_{02}-a_{20}))/a_{02}^2 \\
(3 a_{03} a_{12} (a_{02}-a_{20}))/a_{02}^2&(3 a_{03}^2 (a_{02}-a_{20}))/a_{02}^2},
$$
and
$$
\det\hess\phi(0)=
\dfrac{3 a_{03}^2 (a_{02}-a_{20})^2 (-a_{12}^2+a_{03} a_{21})}{a_{02}^4}.
$$
So, if $a_{03}(-a_{12}^2+a_{03} a_{21})\ne0$, then $c$ at $0$ is an
$S_1^\pm$ singularities.
We assume $-a_{12}^2+a_{03} a_{21}=0$.
Then since we assumed $a_{03}\ne0$, 
it holds that $a_{21}=a_{12}^2/a_{03}$.
Then we see
$f_{uv}=\beta f_{vv}$, where $\beta=a_{12}/a_{03}$ at $0$.
We set $\xi=\partial_u-\beta\partial_v$, $\eta=\partial_v$.
Then $(\xi,\eta)$ is an $SB$-$2$-adapted vector fields. 
Moreover, we set $(\bar\xi,\bar\eta)$ as in Lemma \ref{lem:s3adapted},
where
\begin{align*}
\alpha=&0,\ \beta=a_{12}/a_{03},\\
\alpha_1=&
\dfrac{1}{a_{03}^3 (a_{20}-a_{02})}
\Big(3 a_{02}^3 a_{12}^3 - a_{04} a_{12}^3 + 
 3 a_{12}^2 a_{13} a_{03} \\
&\hspace{30mm}+ (a_{02}^2 a_{12} a_{20} + 2 a_{02} a_{12} a_{20}^2 - 
    3 a_{12} a_{22}) a_{03}^2 + a_{31} a_{03}^3
\Big),\\
\beta_1=&
\dfrac{1}{a_{03}^3}
\Big(-3 a_{02}^3 a_{12}^2 + a_{04} a_{12}^2 - 
 2 a_{12} a_{13} a_{03} + (a_{02}^2 a_{20} - 2 a_{02} a_{20}^2 + a_{22}) a_{03}^2\Big).
\end{align*}
Then $(\bar\xi,\bar\eta)$ is an $S$-$3$-adapted vector fields. 
By a direct calculation, we have
$$\det(\bar\xi c,\bar\xi^3\bar\eta c,\bar\eta^2 c)(0)
=
\dfrac{-2 (a_{20}-a_{02})s}
{a_{02}^2 a_{03}^2}.
$$
This shows the assertion for $S_2$ singularity.
By the above arguments, if $c_u\times c_{vv}\ne0$, then $\phi_{vv}\ne0$ holds.
This shows $B_2^\pm$ singularities do not appear.
Now we assume
$c_u\times c_{vv}=0$, $c_u\times c_{uv}\ne0$, namely,
$a_{03}=0, c_{12}\ne0$.
We set
$$
\eta=
a_{12} v/(a_{02}-a_{20})\partial_u+\partial_v.
$$
Then we see
$\xi c(0)=(*,0,0)$,
$\xi\eta c(0)=(*,*,0)$ and
$\eta^3c(0)=(*,*,0)$.
This shows $\det(\xi c,\xi\eta c,\eta^3c)=0$ at $0$.
This implies $H_2$ singularity does not appear.
\end{proof}

\subsection{Singularities of folded surfaces}
Let $f_0=(u,v,a(u,v))$ be a surface ,where
\begin{align}
\label{eq:monge}
a(u,v)&=\sum_{i+j=2}^5 \dfrac{a_{ij}}{i!j!}u^iv^j,
\end{align}
$a_{11}=0$.
In this representation, the directions of the $u$ and $v$-axes
are the principal directions.
Let $\Pi_\theta\subset\R^3$ be the plane generated by
$(\cos\theta,\sin\theta,0)$ and $(0,0,1)$
passing through the origin.
Let $R_\theta:(\R^3,0)\to(\R^3,0)$ be the fold map with respect to $\Pi_\theta$ defined by
$R_\theta(x,y,z)=A_{-\theta}\circ R_y\circ A_\theta$, where
$$
A_\theta=\pmt{
\cos\theta&-\sin\theta&0\\
\sin\theta& \cos\theta&0\\
         0&          0&1},\quad
R_y(x,y,z)=(x,y^2,z).
$$
Let us consider $f=R_\theta(x,y,z)\circ f_0$.
We call $f$ the {\it folded surface of\/ $f_0$ with 
respect to\/ $\Pi_\theta$}.
By the coordinate change
$(u,v)\mapsto(u \cos\theta + v \sin\theta, v \cos\theta - u \sin\theta)$,
we see $f$ has the form
$f(u,v)=(u,v^2,f_3(u,v))$, and
\begin{align*}
&\coef(f_3,u,v,2)\\
=&\Big((0,0),
((a_{20} \cos^2\theta+a_{02} \sin^2\theta)/2,
(a_{20}-a_{02}) \cos\theta \sin\theta,\\
&\hspace{20mm}(a_{02} \cos\theta^2+a_{20} \sin\theta^2)/2)
\Big).
\end{align*}
Since we are interested in $S_2$ and $B_2^\pm$ singularities, we assume 
$a_{20}=a_{02}$ or $\theta=0$.
We set $\phi(\partial_u,\partial_v)=(f_u,f_v,f_{vv})$.
We assume $\theta=0$.
Then $\det\hess\phi(\xi,\eta)(0)=-a_{21}a_{03}$.
If $a_{21}=0$, $a_{03}\ne0$, then
$(u,v)$ is an $S$-$3$-adapted coordinate system.
Thus by Theorem \ref{thm:cri},
$f$ at $0$ is an $S_2$ singularity if and only if
$a_{31}\ne0$.
If $a_{21}\ne0$, $a_{03}=0$, then 
$(u,v)$ is a $B$-$3$-adapted coordinate system.
Thus by Theorem \ref{thm:cri} as well,
$f$ at $0$ is a $B_2^\pm$ singularity if and only if
$3 a_{05} a_{21}-5 a_{13}^2\ne0$.
We assume $a_{20}=a_{02}$.
Then $(\partial_u,\partial_v)$ is an $SB$-$2$-adapted vector fields.
Setting $\phi=\det(f_u,f_v,f_{vv})$, we have
$$
\hess\phi(0)
=\pmt{
h_{11}&0\\
0&h_{22}},$$
where
\begin{align*}
h_{11}=&-a_{21} \cos^3\theta + (2 a_{12} - a_{30}) \cos^2\theta \sin\theta - (a_{03} - 2 a_{21}) \cos\theta \sin^2\theta - a_{12} \sin^3\theta,\\
h_{22}=&a_{03} \cos^3\theta+3 a_{12} \cos^2\theta \sin\theta+3 a_{21} \cos\theta \sin^2\theta+a_{30} \sin^3\theta.
\end{align*}
We 
set $\theta_s$ one of the solutions of $h_{11}=0$,
and
set $\theta_b$ one of the solutions of $h_{22}=0$.
Since 
$h_{11}=(f_3)_{uuv}(0)$ and
$h_{22}=(f_3)_{vvv}(0)$,
the vector fields $(\partial_u,\partial_v)$ is $S$-$3$-adapted 
if $\theta=\theta_s$, and
it is $B$-$3$-adapted
if $\theta=\theta_b$.
By Theorem \ref{thm:cri},
calculating \eqref{eq:cri1}, we see the condition for
$S_2$ singularity is $r_s\ne0$, where
\begin{align*}
r_s=&-a_{31} \cos^4\theta_s + (3 a_{22} - a_{40}) \cos^3\theta_s \sin\theta_s + 
  3 (-a_{13} + a_{31}) \cos^2\theta_s \sin^2\theta_s\\
&\hspace{5mm}
 + (a_{04} - 3 a_{22}) \cos\theta_s\sin^3\theta_s + 
  a_{13} \sin^4\theta_s.
\end{align*}
Again by Theorem \ref{thm:cri},
calculating \eqref{eq:cri2}, we see the condition for
$B_2^\pm$ singularity is $r_b\ne0$, where $r_b$ is
\begin{align*}
&\Big(5 a_{13}^2 - 3 a_{05} a_{21}\Big) \cos^8\theta_b \\
&
+ \Big(6 a_{05} a_{12} - 10 a_{04} a_{13} - 15 a_{14} a_{21} + 30 a_{13} a_{22} - 
     3 a_{05} a_{30}\Big) \cos^7\theta_b \sin\theta_b \\
&+ \Big(5 a_{04}^2 - 3 a_{03} a_{05} - 30 a_{13}^2 + 30 a_{12} a_{14} + 6 a_{05} a_{21} - 
     30 a_{04} a_{22} + 45 a_{22}^2 - 30 a_{21} a_{23} \\
&\hspace{10mm}- 15 a_{14} a_{30} + 
     30 a_{13} a_{31}\Big) \cos^6\theta_b \sin^2\theta_b \\
&+ \Big(-3 a_{05} a_{12} + 
     5 (-3 a_{03} a_{14} + 6 a_{14} a_{21} - 24 a_{13} a_{22} 
+ 12 a_{12} a_{23} - 
        6 a_{23} a_{30} \\
&\hspace{10mm}+ 6 a_{04} (a_{13} - a_{31}) + 18 a_{22} a_{31} - 6 a_{21} a_{32} + 
        2 a_{13} a_{40})\Big) \cos^5\theta_b \sin^3\theta_b 
\\
&+ 
  5 \Big(9 a_{13}^2 + 6 a_{04} a_{22} - 18 a_{22}^2 - 6 a_{03} a_{23} + 12 a_{21} a_{23} - 
     20 a_{13} a_{31} + 9 a_{31}^2 \\
&\hspace{10mm}- 3 a_{12} (a_{14} - 4 a_{32})
- 6 a_{30} a_{32} - 
     2 a_{04} a_{40} + 6 a_{22} a_{40} - 3 a_{21} a_{41}\Big) \cos^4\theta_b \sin^4\theta_b \\
\end{align*}
\begin{align*}
&+ \Big(90 a_{13} a_{22} - 30 a_{12} a_{23} + 10 a_{04} a_{31} 
- 120 a_{22} a_{31} - 
     30 a_{03} a_{32} + 60 a_{21} a_{32} - 30 a_{13} a_{40} \\
&\hspace{10mm}+ 30 a_{31} a_{40} + 30 a_{12} a_{41} - 
     15 a_{30} a_{41} 
- 3 a_{21} a_{50}\Big) \cos^3\theta_b \sin^5\theta_b \\
&+ \Big(45 a_{22}^2 + 30 a_{13} a_{31} - 30 a_{31}^2 - 30 a_{12} a_{32} - 
     30 a_{22} a_{40} + 5 a_{40}^2 \\
&\hspace{10mm}- 15 a_{03} a_{41} + 30 a_{21} a_{41} + 6 a_{12} a_{50} - 
     3 a_{30} a_{50}\Big) \cos^2\theta_b \sin^6\theta_b \\
&+ \Big(30 a_{22} a_{31} - 10 a_{31} a_{40} 
- 15 a_{12} a_{41} - 3 a_{03} a_{50} + 
     6 a_{21} a_{50}\Big) \cos\theta_b \sin^7\theta_b \\
&+ \Big(5 a_{31}^2 - 3 a_{12} a_{50}\Big) \sin^8\theta_b.
\end{align*}
In \cite{brucewilkin}, generic conditions for $S_2$, $B_2$ singularities 
of folded surfaces are studied and a celebrated duality theorem is obtained.
What we gave here are the concrete conditions of those singularities.
It should be remarked that the condition itself can be
obtained by Lemma \ref{lem:jetcondskbk}, namely, by
\cite[Proposition 2.3]{fhskbk}
since the form of $f$ can easily be reduced into the
form $(u,v^2,f_3(u,v))$.
We remark that considering the resultant of 
$h_{11}=0$ and $r_s=0$ with respect to $\cot\theta_s$,
we obtain a formula written in terms of the coefficients of
the function $f_3(u,v)$.
If $f$ satisfies that the resultant vanishes,
then it implies that there exists $\theta_s$ such that
$h_{11}=r_s=0$.
Namely, a folded surface with respect to the plane $\Pi_{\theta_s}$
is not an $S_2$ singularity.
This condition have some certain geometric meaning for a surface
at an umbilic point, however, 
the resultant is too complicated to find such meaning,
in fact, it has 278 terms.
Moreover, one can also consider the resultant of
$h_{22}=0$ and $r_b=0$ with respect to  $\cot\theta_s$.
However, it is also too complicated to calculate, 
and we do not clarify the meaning.

\newcounter{lemma}
\appendix
\section{Diffeomorphism on the target space}
In this section,
it is discussed
the effect of a diffeomorphism of the target space.
Let $\Phi:(\R^3,0)\to(\R^3,0)$ be a diffeomorphism-germ.
For a vector field $\zeta_1$ on $(\R^3,0)$, we calculate
$\zeta_1\Phi(f)$.
Then we have
$$
\zeta_1\Phi(f)
=
d(\Phi(f))(\zeta_1)
=
d\Phi\circ df(\zeta_1).
$$
Here, the map $d\Phi$ can be interpreted that the map
from the point $p=(x,y,z)\in \R^3$, it produces a linear transformation
on $T_p\R^3$.
We can set $d\Phi$ as a $3\times 3$ matrix valued map
$W(x,y,z)$.
Since $df(\zeta_1)=\zeta_1 f$, it holds that
\begin{equation}\label{eq:target01}
d\Phi\circ df(\zeta_1)=W(f)\zeta_1 f.
\end{equation}
Let $f:(\R^2,0)\to(\R^3,0)$ be a map-germ satisfying $\rank df_0=1$.
A vector field $\zeta$ is said to be {\it null\/} if
$\zeta(0)$ generates the kernel of $df_0$.
\begin{lemma}\label{lem:target}
For a map-germs\/ $f:(\R^2,0)\to(\R^3,0)$ and\/
$\Phi:(\R^3,0)\to(\R^3,0)$,
the differential of\/ $\Phi\circ f$ by vector fields\/
$\zeta_1,\ldots,\zeta_5$ 
satisfy the following.
\renewcommand{\theenumi}{$(\mathrm{T}$-$\arabic{enumi})$}
\begin{enumerate}
\item\label{itm:target01}
The differential by a vector field\/ $\zeta_1$ satisfies\/
$\zeta_1(\Phi\circ f)=W(f) \zeta_1f$ at\/ $0$.
\item\label{itm:target02}
If either of\/ $\zeta_1,\zeta_2$ is null,
then 
the differential by two vector fields\/ $\zeta_1,\zeta_2$ satisfies\/
$\zeta_2\zeta_1(\Phi\circ f)=W(f) \zeta_2\zeta_1f$ at\/ $0$.
\item\label{itm:target03}
If both of\/ $\zeta_1,\zeta_3$ are null,
then 
the differential by three vector fields\/ 
$\zeta_1,\zeta_2,\zeta_3$ satisfies\/
$\zeta_3\zeta_2\zeta_1(\Phi\circ f)=
W(f) \zeta_3\zeta_2\zeta_1f$ at\/ $0$.
\item\label{itm:target04} 
If\/ $f$ is of\/ $SB$-type, and\/ $(\xi,\eta)$ is of\/ $SB$-$2$-adapted,
and
if any of\/ $\zeta_1,\zeta_2,\zeta_3$ are equal to\/ $\eta$, and
others are equal to\/ $\xi$,
then 
the differential by three vector fields\/ 
$\zeta_1,\zeta_2,\zeta_3$ satisfies\/
$\zeta_3\zeta_2\zeta_1(\Phi\circ f)=W(f) \zeta_3\zeta_2\zeta_1f$ at\/ $0$.
\item\label{itm:target05}
If\/ $f$ is of\/ $S$-type, and\/ $(\xi,\eta)$ is of\/ $S$-$3$-adapted,
and
if any of\/ $\zeta_1,\ldots,\zeta_4$ are equal to\/ $\eta$, and
others are equal to\/ $\xi$,
then 
the differential by four vector fields\/ $\zeta_1,\ldots,\zeta_4$ satisfies\/
$\zeta_4\zeta_3\zeta_2\zeta_1(\Phi\circ f)=
W(f) \zeta_4\zeta_3\zeta_2\zeta_1f$ at\/ $0$.
\item\label{itm:target06} 
If\/ $f$ is of\/ $B$-type, and\/ $(\xi,\eta)$ is of\/ $B$-$3$-adapted,
if any three of\/ $\zeta_1,\ldots,\zeta_4$ are equal to\/ $\eta$, and
the another is equal to\/ $\xi$,
then 
the differential by four vector fields\/ $\zeta_1,\ldots,\zeta_4$ satisfies\/
$\zeta_4\zeta_3\zeta_2\zeta_1(\Phi\circ f)=
W(f) \zeta_4\zeta_3\zeta_2\zeta_1f$ at\/ $0$.
\item\label{itm:target07}
If\/ $f$ is of\/ $H$-$2$-adapted, and\/ $(\xi,\eta)$ is of\/ $H$-$2$-adapted,
if all of\/ $\zeta_1,\ldots,\zeta_5$  are equal to\/ $\eta$
then 
the differential by five vector fields\/ $\zeta_1,\ldots,\zeta_5$ satisfies\/
$\zeta_5\zeta_4\zeta_3\zeta_2\zeta_1(\Phi\circ f)=
W(f) \zeta_5\zeta_4\zeta_3\zeta_2\zeta_1f$ at\/ $0$.
\end{enumerate}
\end{lemma}
\begin{proof}
The assertion
\ref{itm:target01} is obvious from \eqref{eq:target01}.
Differentiating \eqref{eq:target01} by the vector fields
$\zeta_2,\zeta_3$, we have
\begin{align}
\label{eq:target02}
\zeta_2\Big(W(f)\zeta_1 f\Big)
=&
\Big(W(f)_x \zeta_2f_1+W(f)_y \zeta_2f_2+W(f)_z \zeta_2f_3\Big)\zeta_1 f
+W(f)\zeta_2\zeta_1 f,
\end{align}
\begin{align}
\label{eq:target03}
&\zeta_3\zeta_2\Big(W(f)\zeta_1 f\Big)\\
=&
\Bigg[
\Big(W(f)_{xx}\zeta_3f_1+W(f)_{xy}\zeta_3f_2+W(f)_{xz}\zeta_3f_3\Big)\zeta_2f_1
+
W(f)_x \zeta_3\zeta_2f_1\nonumber\\
&+
\Big(W(f)_{yx}\zeta_3f_1+W(f)_{yy}\zeta_3f_2+W(f)_{yz}\zeta_3f_3\Big)\zeta_2f_2
+
W(f)_y \zeta_3\zeta_2f_2\nonumber\\
&+
\Big(W(f)_{zx}\zeta_3f_1+W(f)_{zy}\zeta_3f_2+W(f)_{zz}\zeta_3f_3\Big)\zeta_2f_3
+
W(f)_z \zeta_3\zeta_2f_3\Bigg]\zeta_1f\nonumber\\
&+
\Big(W(f)_x \zeta_3f_1+W(f)_y \zeta_3f_2+W(f)_z \zeta_3f_3\Big)
\zeta_2\zeta_1 f
+W(f)\zeta_3\zeta_2\zeta_1 f\nonumber\\
=&
\Bigg(
\sum_{i,j=1}^3W(f)_{x_ix_j}\zeta_3f_i\zeta_2f_j
+\sum_{i=1}^3W(f)_{x_i}\zeta_3\zeta_2f_i
\Bigg)\zeta_1f
+
\Bigg(\sum_{i=1}^3W(f)_{x_i}\zeta_3f_i\Bigg)\zeta_2\zeta_1f\nonumber\\
&+
W(f)\zeta_3\zeta_2\zeta_1f\nonumber\\
=&
\Bigg(
\sum_{i,j=1}^3W(f)_{x_ix_j}\zeta_3f_i\ \zeta_2f_j
\Bigg)\zeta_1f
+
\sum_{i=1}^3W(f)_{x_i}
\Big(\zeta_3\zeta_2f_i\ \zeta_1f+\zeta_3f_i\ \zeta_2\zeta_1f\Big)
\nonumber\\
&
+
W(f)\zeta_3\zeta_2\zeta_1f.\label{eq:target03sigma}
\end{align}
Differentiating \eqref{eq:target03sigma} by $\zeta_4$ again,
we have
\begin{align}
\label{eq:target04}
 &\zeta_4\zeta_3\zeta_2\Big(W(f)\zeta_1 f\Big)\\
=&
\Bigg[
\sum_{i,j=1}^3\Bigg(
\sum_{k=1}^3
W(f)_{x_ix_jx_k}\zeta_4f_k\Bigg)\zeta_3f_i\ \zeta_2f_j
+
W(f)_{x_ix_j}\Big(\zeta_4\zeta_3f_i\ \zeta_2f_j
+\zeta_3f_i\ \zeta_4\zeta_2f_j\Big)
\Bigg]\zeta_1f\nonumber\\
&+
\Bigg(
\sum_{i,j=1}^3W(f)_{x_ix_j}\zeta_3f_i\ \zeta_2f_j
\Bigg)\zeta_4\zeta_1f\nonumber\\
&+
\sum_{i=1}^3\Bigg[\sum_{k=1}^3W(f)_{x_ix_k}\zeta_4f_k
\Big(\zeta_3\zeta_2f_i\ \zeta_1f+\zeta_3f_i\ \zeta_2\zeta_1f\Big)\nonumber\\
&+
W(f)_{x_i}
\Big(\zeta_4\zeta_3\zeta_2f_i\ \zeta_1f+\zeta_3\zeta_2f_i\ \zeta_4\zeta_1f
+\zeta_4\zeta_3f_i\ \zeta_2\zeta_1f+\zeta_3f_i\ \zeta_4\zeta_2\zeta_1f
\Big)\Bigg]\nonumber\\
&
+
\Bigg(
\sum_{i=1}^3W(f)_{x_i}\zeta_4f_i
\Bigg)\zeta_3\zeta_2\zeta_1f
+
W(f)\zeta_4\zeta_3\zeta_2\zeta_1f.\nonumber
\end{align}
The assertion
\ref{itm:target02} follows from \eqref{eq:target02}.
For the assertion \ref{itm:target03},
by \eqref{eq:target03}, if 
$\zeta_1,\zeta_3$ are null, then all the terms except for
$W(f)\zeta_3\zeta_2\zeta_1f$ vanish.
Thus the assertion \ref{itm:target03} follows.
For \ref{itm:target04}, 
we assume one of $\zeta_i$ is $\eta$, and others are $\xi$.
Then by \eqref{eq:target03},
all the terms except for $W(f)\zeta_3\zeta_2\zeta_1f$
vanish since such terms have
$\eta f$, $\eta\xi f$ or $\xi\eta f$, and by the assumption
$SB$-$2$-adaptivity.
For \ref{itm:target05}, 
we assume one of $\zeta_i$ is $\eta$, and others are $\xi$.
Then by \eqref{eq:target04},
all the terms except for $W(f)\zeta_4\zeta_3\zeta_2\zeta_1f$
vanish since such terms have
$\eta f$, $\eta\xi f$, $\xi\eta f$, $\xi\eta^2 f$, $\eta\xi\eta f$ 
or $\eta^2\xi f$, and by the assumption
$S$-$3$-adaptivity.
For \ref{itm:target06}, 
we assume three of $\zeta_i$ are $\eta$, and another is $\xi$.
Then by \eqref{eq:target04},
all the terms except for $W(f)\zeta_4\zeta_3\zeta_2\zeta_1f$
vanish since such terms have
$\eta f$, $\eta\xi f$, $\xi\eta f$ or $\eta^3 f$,
and by the assumption
$B$-$3$-adaptivity.
For 
\ref{itm:target07}, taking a differential of \eqref{eq:target04} by
$\eta$.
Under the assumption $H$-$2$-adaptivity,
looking for a term in \eqref{eq:target04} that is not multiplied
by more than $2$ of terms that are $0$,
we obtain
$\sum_{i=1}^3\zeta_4\zeta_3\zeta_2f_i\ \zeta_1f$,
$\sum_{i=1}^3\zeta_3f_i\ \zeta_4\zeta_2\zeta_1f$,
$\sum_{i=1}^3\zeta_4f_i\ \zeta_3\zeta_2\zeta_1f$.
Differentiating all these terms by $\eta$, one can see that they
vanish under the assumption $H$-$2$-adaptivity.
Thus all of the terms vanish except for
$W(f)\zeta_5\zeta_4\zeta_3\zeta_2\zeta_1f$.
We remark that $\zeta_5(W(f))$ vanishes since $\zeta_5f=0$ at $0$.
This shows the assertion.
\end{proof}

\chartonoff{
\section{Recognition chart}
Since the criteria we gave are complicated a little.
We give a recognition chart for the all singularities
up to codimension (codim) two appearing on surfaces.
We assume $f:(\R^2,0)\to(\R^3,0)$ be a map-germ at $0$,
and $0$ is a singular point.
The enclosed item by
\doublebox{\hspace{5mm}}
means the chart ends there,
and
the item \Ovalbox{\hspace{5mm}}
means that branch will continue thereafter.
\clearpage
\subsection{The first step}

\begin{figure}[h!]
\includegraphics[width=\linewidth]{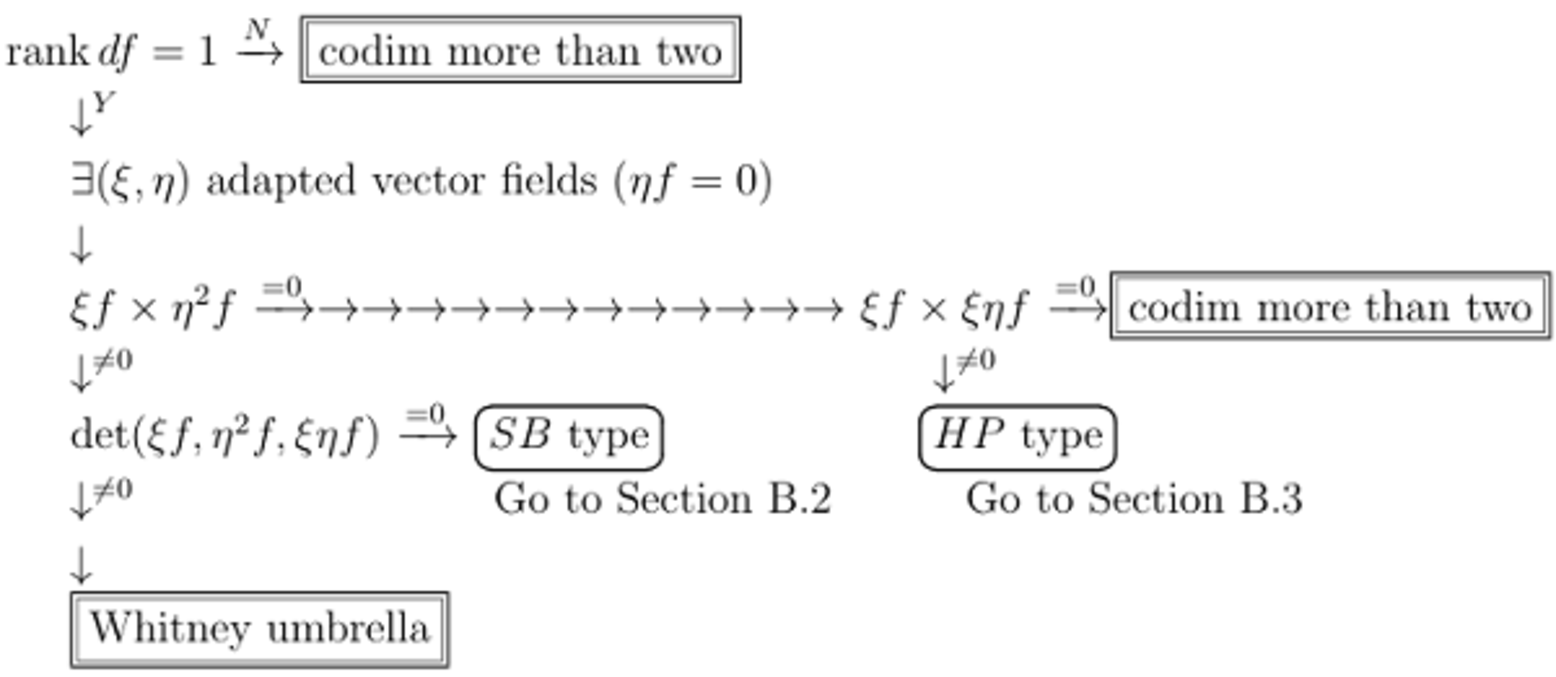}
\end{figure}

\subsection{$SB$ type}\label{sec:chartsbtype}

\begin{figure}[h!]
\includegraphics[width=\linewidth]{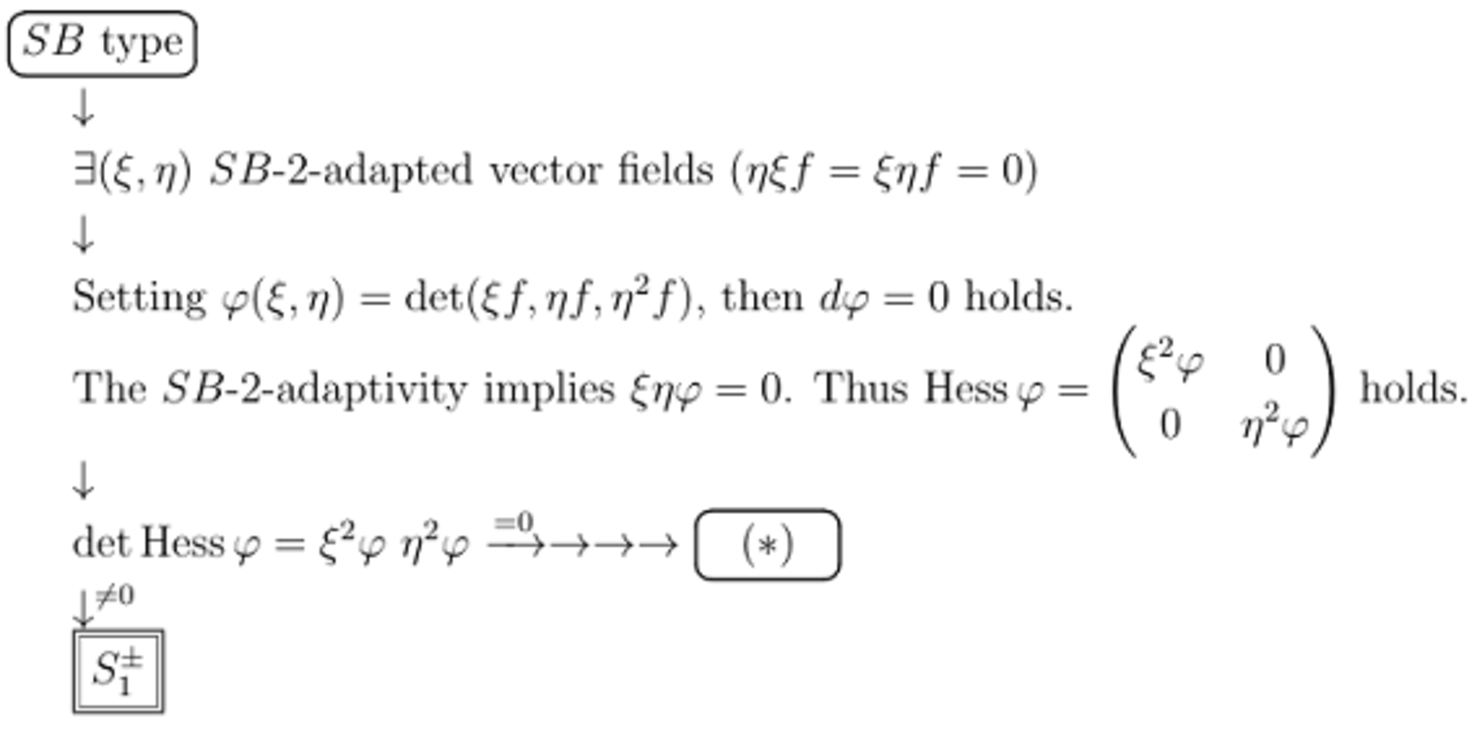}
\end{figure}
\begin{figure}[h!]
\includegraphics[width=\linewidth]{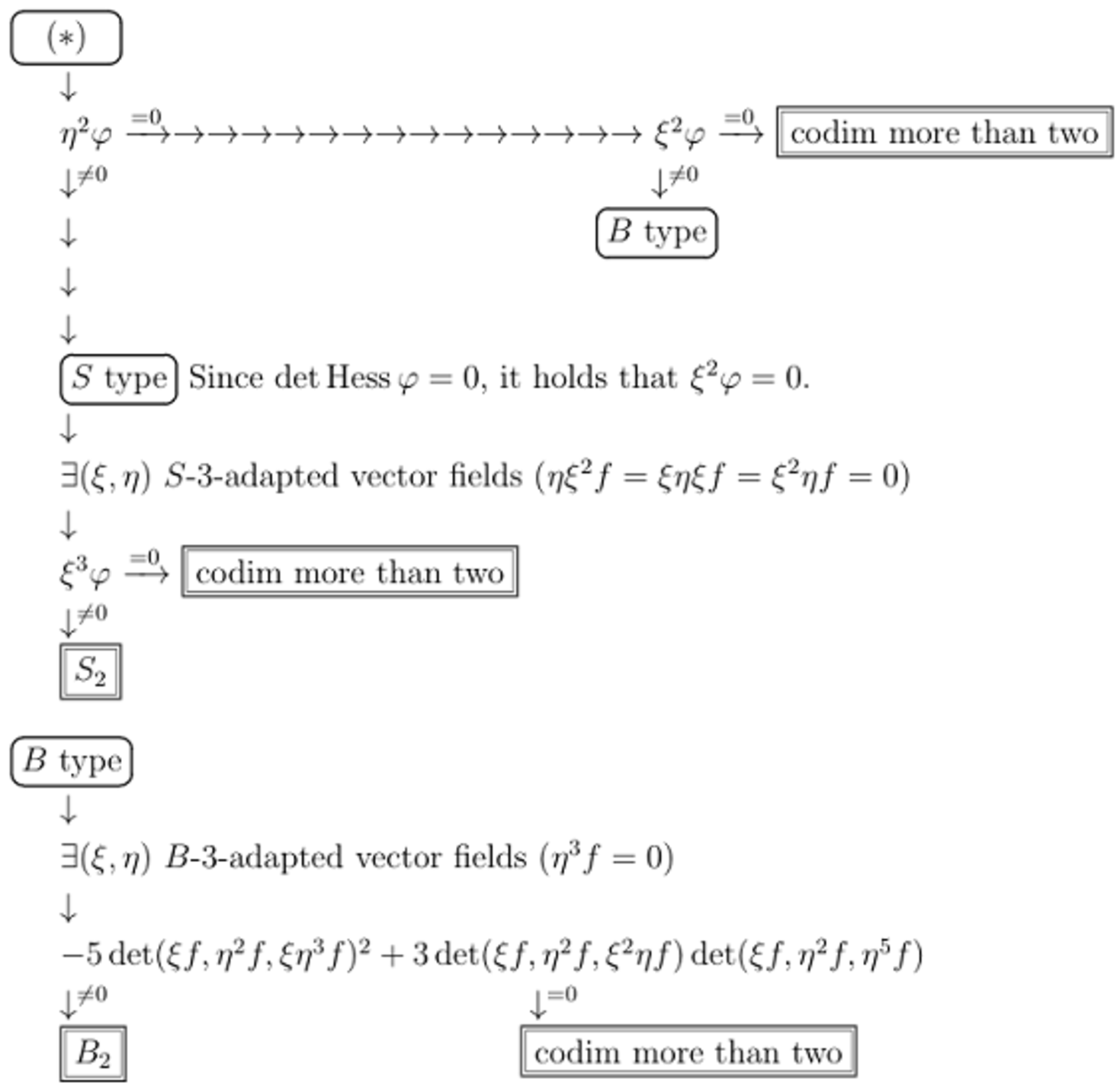}
\end{figure}
\clearpage

\subsection{$HP$ type}\label{sec:charthtype}
\begin{figure}[h!]
\includegraphics[width=.7\linewidth]{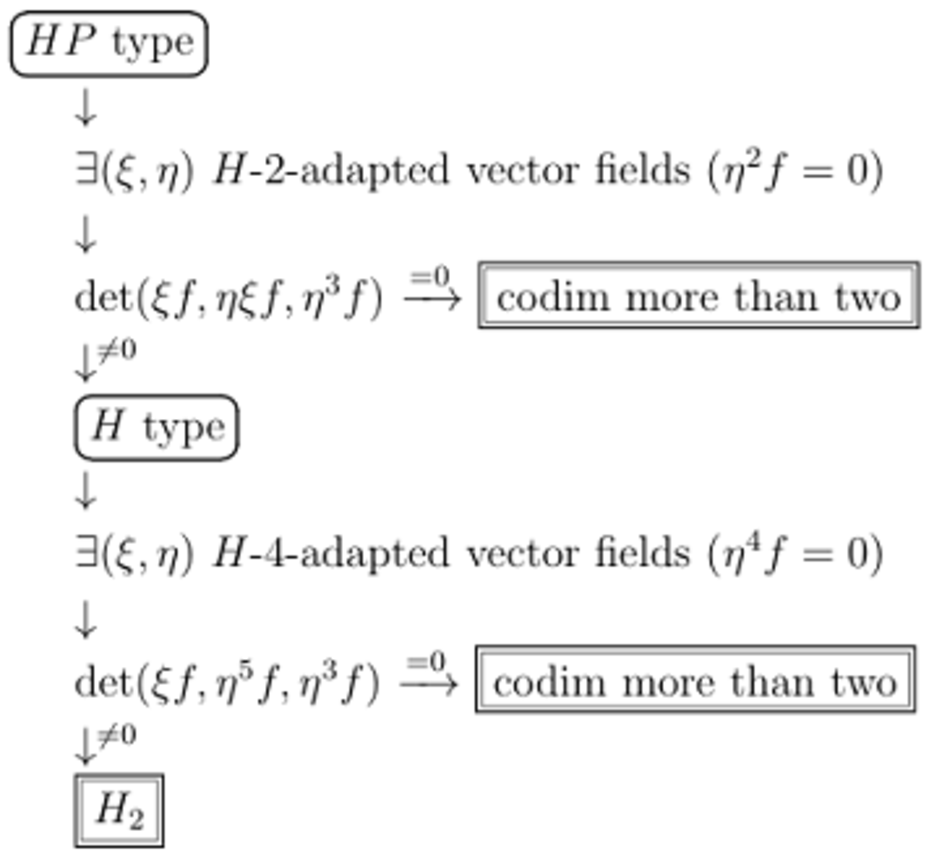}
\end{figure}

}

\section*{Acknowledgments}
Partly supported by the
JSPS KAKENHI Grant Numbers 22K03312, 22KK0034, 25K 07001.


\toukoudel{
\medskip
{\footnotesize
\begin{flushright}
\begin{tabular}{l}
(Saji and Shimada)\\
Department of Mathematics,\\
Graduate School of Science, \\
Kobe University, \\
1-1, Rokkodai, Nada, Kobe \\
657-8501, Japan\\
E-mail: {\tt saji@math.kobe-u.ac.jp}\\
E-mail: {\tt 231s010s@stu.kobe-u.ac.jp}
\end{tabular}
\end{flushright}}
}


\begin{thebibliography}{99}
\bibitem{brucewilkin}
J. W. Bruce and T. C. Wilkinson,
{\it Folding maps and focal sets},
Singularity theory and its applications, Part I (Coventry, 1988/1989), 63--72,
Lecture Notes in Math., {\bf 1462}, Springer, Berlin, 1991.


\bibitem{chenmatumoto}
X.-Y. Chen and T. Matumoto,
{\it On generic\/ $1$-parameter families of\/ $C^\infty$-maps of
 an\/ $n$-manifold into a\/ $(2n-1)$-manifold}, 
Hiroshima Math. J., {\bf 14} (1985), no. 3, 547--550.

\bibitem{fhskbk}
T. Fukui and M. Hasegawa,
{\it 
Distance squared functions on singular surfaces
parameterized by smooth maps\/ $\A$-equivalent to\/ $S_k$,
$B_k$, $C_k$ and\/ $F_4$},
preprint.

\bibitem{furuvra}
H. Furuhata and L. Vrancken,
{\it The center map of an affine immersion},
Results in Math.,  {\bf 49} (2006), 201--217.

\bibitem{gray}
A. Gray,
{\it Modern differential geometry of curves and surfaces with Mathematica},
Second edition. CRC Press, Boca Raton, FL, 1998.


\bibitem{irrtbook}
S. Izumiya, M. C. Romero Fuster, M. A. S. Ruas and F. Tari,
{\it Differential geometry from a singularity theory viewpoint},
World Scientific Publishing Co. Pte. Ltd., Hackensack, NJ, 2016.


\bibitem{kobanomi}
S. Kobayashi and K. Nomizu,
{\it Foundations of Differential Geometry Volume\/ {\rm I}},
Wiley Interscience (1964).

\bibitem{KRSUY}
M. Kokubu, W. Rossman, 
K. Saji, M. Umehara and K. Yamada,   
{\it Singularities of flat fronts in hyperbolic\/ $3$-space}, Pacific J. 
Math. {\bf 221}  (2005), 303--351.

\bibitem{mond}
D. Mond, {\it On the classification of germs of maps 
from\/ $\R^2$ to\/ $\R^3$},
Proc. London Math. Soc. {\bf 50} (1985), no. 2, 333--369.

\bibitem{sajipl}
K. Saji, {\it Criteria for singularities of smooth maps from 
the plane into the plane and their applications}, 
Hiroshima Math. J. {\bf 40} (2010), no. 2, 229--239.

\bibitem{sk}
K. Saji, {\it Criteria for cuspidal\/ $S_k$ singularities 
and their applications}, 
J. G\"okova Geom. Topol. GGT {\bf 4} (2010), 67--81. 

\bibitem{ak}
K. Saji, M. Umehara and K. Yamada,
{\it $A_k$ singularities of wave fronts},
Math. Proc. Cambridge Philos. Soc. {\bf 146} (2009), no. 3, 731--746.


\bibitem{usybook}
M. Umehara, K. Saji and K. Yamada,
{\it Differential geometry of curves and surfaces with singularities},
Algebraic and Differential Geometry, $1$.
World Scientific Publishing Co. Pte. Ltd., Hackensack, NJ, 2022.

\bibitem{whitney}
H. Whitney, {\it The general type of singularity of a set of\/ $2n-1$ 
smooth functions of\/ $n$ variables}, Duke Math. J. {\bf 10} (1943), 161--172.
\end{thebibliography}
\end{document}